\documentclass{amsart}
\usepackage[utf8]{inputenc}
\usepackage{amssymb,latexsym}
\usepackage{amsmath}
\usepackage{graphicx}
\usepackage{hyperref}
\usepackage{textcomp}
\usepackage[dvipsnames]{xcolor}
\usepackage{tabularx}
\usepackage{array}

\usepackage{amsthm,amssymb,enumerate,graphicx, tikz}
\usepackage{amscd}
\usepackage{setspace}
\usepackage{comment}
\usepackage{hyperref}
\usepackage{cleveref}
\usepackage{caption}
\usepackage{eucal}

\def\@logofont{\footnotesize}
\def\@setaddresses{\par
  \nobreak \begingroup
  \footnotesize
  \def\author##1{\nobreak\addvspace\bigskipamount}%
  \def\\{\par\nobreak}%
  \interlinepenalty\@M
  \def\address##1##2{\begingroup
    \par\addvspace\bigskipamount\indent
    \@ifnotempty{##1}{(\ignorespaces##1\unskip) }%
    {\scshape\ignorespaces##2}\par\endgroup}%
  \def\curraddr##1##2{\begingroup
    \@ifnotempty{##2}{\nobreak\indent\curraddrname
      \@ifnotempty{##1}{, \ignorespaces##1\unskip}\/:\space
      ##2\par}\endgroup}%
  \def\email##1##2{\begingroup
    \@ifnotempty{##2}{\nobreak\indent\emailaddrname
      \@ifnotempty{##1}{, \ignorespaces##1\unskip}\/:\space
      \ttfamily##2\par}\endgroup}%
  \def\urladdr##1##2{\begingroup
    \def~{\char`\~}%
    \@ifnotempty{##2}{\nobreak\indent\urladdrname
      \@ifnotempty{##1}{, \ignorespaces##1\unskip}\/:\space
      \ttfamily##2\par}\endgroup}%
  \addresses
  \endgroup
}
\renewcommand*\subjclass[2][2010]{%
  \def\@subjclass{#2}%
  \@ifundefined{subjclassname@#1}{%
    \ClassWarning{\@classname}{Unknown edition (#1) of Mathematics
      Subject Classification; using '2000'.}%
  }{%
    \@xp\let\@xp\subjclassname\csname subjclassname@#1\endcsname
  }%
}

\usepackage{graphicx,amsmath,amssymb}
\usepackage{algorithm}
\usepackage{mathdots, comment}
\usepackage[noend]{algpseudocode}

\newtheorem{theorem}{Theorem}[section]

\newtheorem*{theorem*}{Theorem}
\newtheorem{proposition}[theorem]{Proposition}
\newtheorem{lemma}[theorem]{Lemma}

\newtheorem{corollary}[theorem]{Corollary}

\theoremstyle{definition}
\newtheorem{definition}[theorem]{Definition}

\theoremstyle{remark}
\newtheorem{remark}[theorem]{Remark}
\newtheorem{example}[theorem]{Example}

\begin{document}
\title[Matchings and product growth]{Matchings and product growth in modular abelian independence groups}

\author[M. Aliabadi, J. Losonczy]{Mohsen Aliabadi$^{1}$ \and Jozsef Losonczy$^{2,*}$}
\thanks{$^1$Department of Mathematics, Clayton State University, 
2000 Clayton State Boulevard, Lake City, Georgia 30260, USA.  \url{maliabadi@clayton.edu}.\\
$^2$Department of Mathematics, Long Island University,
720 Northern Blvd, Brookville, New York 11548, USA. \url{Jozsef.Losonczy@liu.edu}.}
\thanks{$^*$Corresponding Author.}


\begin{abstract}
\begin{sloppypar}
We unify two matching theories, one for finite subsets of groups and the other for finite-dimensional subspaces in a field extension. To achieve this, we study groups equipped with a compatible finitary matroid structure, termed here independence groups.  Applying Rado’s independent transversal theorem, we derive necessary and sufficient rank criteria for matchability between finite-rank sets. In the setting of a modular abelian independence group \( G \), we develop an analogue of the $e$-transform from additive number theory, derive structural matching criteria, and characterize a global matching property by the absence of a finite-rank submonoid \(H\) satisfying 
\[ 1<\rho(H)<\rho(G), 
\] 
where \( \rho \) denotes rank. Examples of modular abelian independence groups are given and examined in the matching context.  Arising from this matching theory, but formulated without any reference to it, is a product-growth bound that generalizes the Cauchy--Davenport theorem: we define a parameter \(\mu(G)\) and prove that 
\[
\rho(XY)
\geq \min\{\mu(G),\rho(X)+\rho(Y)-1\}
\]
for all nonempty finite-rank subsets \( X, Y \) of \( G \). Furthermore, \( \rho(XY) \) is shown to be controlled from below by a submonoid of \( G \) that stabilizes a flat, a phenomenon reminiscent of Kneser's theorem.  
\end{sloppypar}
\end{abstract}

\maketitle

\section{Introduction}\label{Intro}

Let \(G\) be a group, written multiplicatively, and let \(A,B\subseteq G\) be finite subsets with \(|A|=|B|\). A \emph{matching} from \(A\) to \(B\) is a bijection \(f\colon A\to B\) such that \(af(a)\notin A\) for every \(a\in A\). Necessarily, \(1\notin B\); otherwise, \(f(a)=1\) for some \(a\in A\), and then \(af(a)=a\in A\). The theory of matchings in groups was initiated in \cite{Losonczy 1}, motivated by a problem of Wakeford concerning canonical forms for homogeneous polynomials \cite{Wakeford}. More specifically, Wakeford's problem is to determine the sets of monomials that can be removed from a generic homogeneous polynomial over \( \mathbb{C} \) by a linear change in variables.  

To see the connection with matchings, let \( A \) be a set of monomials in \( q \) variables, where \( |A| = q(q-1) \) and the monomials in \( A \) all have the same total degree.  Identify the monomials with their exponent vectors and let 
\[ 
B = \{ e_i - e_j : 1 \leq i \neq j \leq q \},
\] 
where \( e_1,\ldots ,e_q \) are the standard unit vectors.  In \cite{Losonczy 1}, a weighted bipartite graph is associated to \( A \); we denote its biadjacency matrix here by \( \mathcal{M}_A \).  It is then shown, using the notion of apolarity, that the monomials in \( A \) are removable precisely when \( \det \mathcal{M}_A \neq 0 \). Furthermore, the nonzero terms in the Leibniz expansion of \( \det \mathcal{M}_A \) correspond to (additive) matchings \( f : A \to B \).  If each monomial in \( A \) involves all \( q \) variables, then there exists an \emph{acyclic} matching \( f \), uniquely determined by the multiset of sums \( a + f(a) \) (\( a \in A \)). This uniqueness property ensures that the corresponding term in the determinant cannot be canceled. 

The above connection with removable sets of monomials attaches direct algebraic significance to the existence of matchings and motivates their further study.  In this paper, we show that the payoff includes a far-reaching generalization of the classical Cauchy--Davenport theorem and also of the $e$-transform from additive number theory.

A group \(G\) is said to have the \emph{matching property} if every pair of finite subsets \(A,B\) satisfying \(|A|=|B|\) and \(1\notin B\) admits a matching from \(A\) to \(B\). The groups having the matching property were first classified in the abelian setting in \cite{Losonczy 2}: these are precisely the groups that are torsion-free or of prime order. This result was subsequently extended to arbitrary groups in \cite{Eliahou 1}. Other work has addressed the enumeration of matchings, acyclic matchings, structural criteria for matchability, and decomposition theorems for unmatchable pairs \cite{Aliabadi 0,Aliabadi,Aliabadi 5,AliabadiLosonczy,Alon,Hamidoune}.

A parallel matching theory for finite-dimensional subspaces in a field extension was introduced by Eliahou and Lecouvey~\cite{Eliahou 2}. Let $K\subseteq L$ be a field extension, and let $A,B\subseteq L$ be $n$-dimensional $K$-subspaces. An ordered basis $\{a_1,\ldots,a_n\}$ of $A$ is \emph{matched} to an ordered basis $\{b_1,\ldots,b_n\}$ of $B$ if
\[
a_i^{-1}A \cap B
\,\subseteq\,
\operatorname{span}_K\{b_j : j \neq i\}
\quad \text{for all }i=1,2,\ldots,n.
\]
The subspace $A$ is \emph{matched} to $B$ if every ordered basis of $A$ can be matched to some ordered basis of $B$. 

The matching setups described above, for finite subsets of groups and finite-dimensional subspaces of field extensions, give rise to a common structural problem: to determine how multiplication and linear independence constrain admissible pairings \( (A,B) \), and to identify the invariant substructures responsible for the failure of matchability.

A matroidal approach to this topic was presented in \cite{AliabadiZerbib}, where matchings are considered between matroids whose ground sets are contained in an abelian group. The framework developed here, although also involving ideas from matroid theory, is different: we place a single translation-invariant independence structure on the entire group.  The compatibility between the independence structure and group operation makes it possible to work systematically with rank, closure, flats, and multiplicatively invariant substructures. This leads to structural matching criteria and product growth theorems which are not available for arbitrary matroids embedded in a group.

Thus we define an independence group to be a group \(G\) equipped with an independence structure \(\mathcal I\) such that \( gA \in \mathcal{I} \) for all \( g \in G \) and \( A \in \mathcal{I} \). Its rank function and closure operator are denoted by \(\rho\) and \(\textup{cl}\), respectively. By taking \( \mathcal{I} \) to be the universal independence structure, in which every subset of \( G \) is independent, we recover the original matching theory for finite subsets of groups. If \(K\subseteq L\) is a field extension and we let \( G \) be the multiplicative group \(L^\times\), then we retrieve the matching theory arising from \cite{Eliahou 2} by taking \( \mathcal{I} \) to be the collection of all subsets of \( G \) that are linearly independent over \( K \). 

In addition to the above two examples, we construct independence groups from the projective groups
\[
L^\times/K^\times,
\]
from coset partitions, and from the cycle spaces of finite simple graphs; see Propositions~\ref{projective independence group} and \ref{coset partition proposition}, and Example~\ref{cycle space independence example}.  In Example~\ref{projective C21 example}, we show that two different independence structures on the same abstract group can produce different matching behavior.  For each family of independence groups considered, we give formulas for the rank and closure, and we show that the structure \( \mathcal{I} \) has the property of being modular. 

Our first results on matchings are derived from Rado's theorem on independent transversals \cite{Rado1942,Mirsky}, the matroidal analogue of Hall's marriage theorem \cite{Hall}.  In an independence group \( G \), let \(A\) have finite rank \(n\), let \(\mathcal A\) be a basis of \(A\), and let \(\mathcal B\) be an independent set of cardinality \(n\). Theorem~\ref{specialized} shows that \(\mathcal A\) is matchable to \(\mathcal B\), relative to \(A\), if and only if
\[
\rho\left(
\bigcap_{a\in J}(a^{-1}A\cap\mathcal B)
\right)
\leq n-|J|
\]
for every \(J\subseteq\mathcal A\). Theorem~\ref{existence condition} then gives the corresponding criterion for the matchability of \( \mathcal{A} \) to every basis of
a set \(B\) of rank \(n\).  

In order to develop a structure theory for matchings in independence groups, we introduce a generalization of the $e$-transform from additive number theory (Theorem \ref{Dyson G}). Using this tool, which requires that the independence group be abelian and modular, we first give a structural characterization of matchability.  If \(A\) and \(B\) are flats of rank \(n\), then Theorem~\ref{simpler characterization} states that \(A\) is matched to \(B\) if and only if
\[
\rho(S)+\rho(R)\leq n
\]
whenever \(S \subseteq A\) and \( R \subseteq B \) are nonempty flats satisfying \( SR = S \). With this result, one easily recovers all previously known structural matching criteria for finite subsets of abelian groups and for finite-dimensional subspaces in field extensions.  

We next consider global matching properties. A flat of finite rank is shown to be self-matched precisely when it does not contain the group identity element
(Theorem~\ref{self matching criterion}). More generally, we say that an independence group has the \emph{matching property} if every pair \((A,B)\) of finite-rank flats satisfying \(\rho(A)=\rho(B) \) and \(1\notin B\) is matched.  Theorem~\ref{MMP} shows that a modular abelian independence group \( G \) has the matching property if and only if it contains no finite-rank submonoid \(H\) satisfying
\[
1<\rho(H)<\rho(G).
\]

This structural obstruction also controls product growth from below. For an independence group \(G\), we let \(\mu(G)\) denote the smallest value of \( \rho(H) \), where \( H \) runs over all finite-rank submonoids of \( G \) satisfying \( \rho(H) > 1 \), with \(\mu(G)=\infty\) when no such submonoid exists. Our main product result, Theorem~\ref{threshold product growth}, states that
\[
\rho(XY)
\,\geq\,
\min\bigl\{
\mu(G),\rho(X)+\rho(Y)-1
\bigr\}
\]
for all nonempty finite-rank subsets \(X,Y\) of \( G \). It also provides a structural explanation for every failure of expected rank growth: if
\[
\rho(XY)<\rho(X)+\rho(Y)-1,
\]
then \(\textup{cl}(XY)\) contains a nonempty flat \(S\) stabilized by a finite-rank flat submonoid \(M\), in the sense that \( SM=S \), with \( \rho(M) > 1\). In this we see a rank-theoretic form of the stabilizer phenomenon underlying classical sum-set results such as the Cauchy--Davenport and Kneser theorems \cite{Davenport 1, Davenport 2, Kneser1953}.

Our proof of Theorem~\ref{threshold product growth} relies on the generalized $e$-transform (Theorem~\ref{Dyson G}). We emphasize that although Theorems~\ref{Dyson G} and \ref{threshold product growth} arise from our matching theory, these results are stated without any reference to it.  It is also worth noting that the Cauchy--Davenport inequality is easily recovered from Theorem~\ref{threshold product growth} (see Remark~\ref{recover Cauchy davenport}).

In Corollary~\ref{low rank matching}, it is shown that every admissible pair of flats with rank less than \( \mu(G) \) is matchable.  The threshold \( \mu(G) \) is computed explicitly for two principal families. For the projective independence group associated with a finite field extension \(K\subsetneq L\), Corollary~\ref{projective threshold and matching property} gives
\[
\mu(G)
=
\min\bigl\{
[F:K]:
K\subsetneq F\subseteq L
\text{ is an intermediate field}
\bigr\}.
\]
It follows that a projective independence group has the matching property precisely when the extension \(K \subsetneq L\) has no nontrivial proper intermediate field.

For the coset-partition independence group \((G,\mathcal I_H)\) associated with a subgroup \(H \) of \(G\), where \(G\) is abelian and \(G/H\) is finite, Proposition~\ref{coset partition threshold} gives
\[
\mu(G)
=
\min\bigl\{
|Q|:Q\text{ is a nontrivial subgroup of }G/H
\bigr\}
\]
when \( G \neq H \), and \( \mu(G) = \infty \) otherwise. Thus \(\mu(G)\) is the smallest prime divisor of \(|G/H|\) if \( G \neq H \), and \(G\) has the matching property precisely when \(G/H\) is trivial or has prime order. In
Example~\ref{coset threshold example}, a nonprojective independence group of this type is exhibited, and it is shown that the bound in Theorem~\ref{threshold product growth} is achieved for a certain pair of sets. 

Finally, Corollary~\ref{cauchy davenport equivalence} shows that the matching property holds provided the inequality
\[
\rho(XY)
\,\geq\,
\min\bigl\{
\rho(G),\rho(X)+\rho(Y)-1
\bigr\}
\]
is satisfied for all nonempty finite-rank subsets \(X,Y\) of \( G \). 

The paper is organized as follows. Section~\ref{Prelims} reviews finitary matroids, rank, and closure, and then introduces independence groups. A definition of matching for independence groups is formulated and reconciled with preexisting notions in Section~\ref{MatchedSets}. Section~\ref{transversal section} establishes general Rado-type matching criteria. Section~\ref{ModSec} develops a theory of modular independence groups and constructs the principal examples. In Section~\ref{e-transform and char sec}, a generalized $e$-transform is introduced and used to establish structural characterizations of matchability. Finally, Section~\ref{self} treats self-matching, characterizes the matching property, establishes the product-growth theorem, and computes the submonoid-rank threshold for projective and coset-partition independence groups.

\section{Preliminaries} \label{Prelims}

In this section we establish the basic setup for our theory.  The main object of interest will be a group with a compatible finitary matroid structure. We provide below all of the definitions and foundational results that are needed to get started.  Additional definitions will be given later, as needed.

We denote the power set of any set \( E \) by \( 2^E \).  If \( K \) is a field, we write \( K^{\times} \) for its multiplicative group \( K \setminus \{ 0 \} \).

A \emph{finitary matroid} is a pair \( (E, \mathcal{I}) \) consisting of a set \( E \) and a nonempty collection \( \mathcal{I} \) of subsets of \( E \) satisfying the following three axioms:
\begin{itemize}
\item[(I1)] If \( A \in \mathcal{I} \) and \( B \subseteq A \), then \( B \in \mathcal{I} \). 
\item[(I2)] If \( A \) and \( B \) are finite elements of \( \mathcal{I} \) and \( |B| = |A| + 1 \), then there exists \( x \in B \setminus A \) such that \( \{ x \} \cup A \in \mathcal{I} \).
\item[(I3)] If \( A  \subseteq E \) and every finite subset of \( A \) belongs to \( \mathcal{I} \), then \( A \in \mathcal{I} \).
\end{itemize}
The collection \( \mathcal{I} \) is called an \emph{independence structure} on \( E \). The elements of \( \mathcal{I} \) are the \emph{independent sets}, and \( E \) is the \emph{ground set}.  If \( A \subseteq E \) and \( A \notin \mathcal{I} \), we call \( A \) a \emph{dependent set}.  A \emph{loop} is an element \( x \) of \( E \) such that \( \{ x \} \) is a dependent set.  We remark that finitary matroids are also known as \emph{independence spaces}. 

We call (I1) the \emph{hereditary axiom}, (I2) the \emph{exchange axiom}, and (I3) the \emph{finite-character axiom}.

\begin{example} \label{universal structure}
Let \( E \) be any set and take \( \mathcal{I} = 2^E \).  Then the pair \( (E,\mathcal{I}) \) is a finitary matroid.  We call \( \mathcal{I} \) the \emph{universal structure} on \( E \).  
\end{example}

\begin{example} \label{linear pi}
Let \( K \subseteq L \) be a field extension and let \( E \subseteq L \).  Define \( \mathcal{I} \) to be the collection of all subsets \( S \) of \( E \) such that \( S \) is linearly independent over \( K \). Then \( (E,\mathcal{I}) \) is a finitary matroid.  Axiom (I2) holds on account of the Steinitz exchange lemma from linear algebra, while (I3) follows immediately from the definition of linear independence.
\end{example}

For any finitary matroid \( (E,\mathcal{I}) \) and any \( A \subseteq E \), we define 
\[ 
\mathcal{I}|_A = \{ B \in \mathcal{I} : B \subseteq A \}. 
\]  
The pair \( (A,\mathcal{I}|_A) \) is then also a finitary matroid.  

A \emph{basis} of a finitary matroid \( (E,\mathcal{I}) \) is an independent set \( \mathcal{A} \) which is maximal in \( \mathcal{I} \) with respect to set-theoretic inclusion. Sometimes we speak of a basis \( \mathcal{A} \) of a subset \( A \) of \( E \).  When we do, we mean that \( \mathcal{A} \) is a basis of \( (A, \mathcal{I}|_A) \).

The following result gathers together some useful, well-known properties of bases.

\begin{proposition}  \label{basis properties}
Let \( (E,\mathcal{I}) \) be a finitary matroid and let \( A \) be a subset of \( E \). Then the following statements hold:
\begin{itemize}
\item[(i)] There exists at least one basis of \( A \).
\item[(ii)]  Any two bases of \( A \) have the same cardinality.
\item[(iii)]  Let \( D \) be an independent subset of \( A \), and let \( \mathcal{A} \) be a basis of \( A \). Then there exists a basis \( \mathcal{A}^* \) of \( A \) such that \( D \subseteq \mathcal{A}^* \subseteq D \cup \mathcal{A} \).
\item[(iv)] Let \( D \) be an independent subset of \( A \).  Then \( D \) can be extended to a basis of \( A \), i.e., there exists a basis \( \mathcal{A} \) of \( A \) such that \( D \subseteq \mathcal{A} \).
\end{itemize}
\end{proposition}

\begin{proof}
Parts (i) through (iii) are proved in Chapter 7 of \cite{Mirsky}, and (iv) is an immediate consequence of (i) and (iii).
\end{proof}

If \( (E, \mathcal{I}) \) is a finitary matroid, we define the \emph{rank} \( \rho(A) \) of any \( A \subseteq E \) to be the cardinal number of any basis of \(A\).  We say that \( A \) has \emph{finite rank} if the cardinal \( \rho(A) \) is a nonnegative integer.  Sometimes we use the shorthand \( \rho(A) < \infty \) to express that \( A \) has finite rank.

Although our principal results concern sets of finite rank, such sets need not be finite.  Accordingly, we allow the ground set to be infinite rather than restrict our work to the setting of finite matroids.

The proposition below contains a pair of important properties of the rank.  Standard sources in the literature often do not treat the transfinite case, so we include a brief proof here.

\begin{proposition}  \label{rank properties}
Let \( (E, \mathcal{I}) \) be a finitary matroid and let \( A,B \) be subsets of \( E \).  Then:
\begin{itemize}
\item[(i)] \( \rho(A) \leq \rho(B) \) if \( A \subseteq B \). 
\item[(ii)] \( \rho(A) + \rho(B) \,\geq\, \rho(A \cup B) + \rho(A \cap B) \). 
\end{itemize}
\end{proposition}

\begin{proof}
We rely heavily on Proposition~\ref{basis properties}. For (i), if \( A \subseteq B \), we choose a basis \( \mathcal{A} \) of \( A \) and extend it to a basis \( \mathcal{B} \) of \( B \).  Then 
\[ 
\rho(A) = |\mathcal{A}| \leq |\mathcal{B}| = \rho(B). 
\] 
For (ii), we choose a basis \( \mathcal{C} \) of \( A \cap B \), extend it to a basis \( \mathcal{C}_A \) of \( A \), and then extend this to a basis \( \mathcal{C}_{A\cup B} \) of \( A\cup B \).  Using these bases, we compute
\begin{align*}
\rho(A\cup B) + \rho(A\cap B) &= |\mathcal{C}_{A\cup B}| + |\mathcal{C}|\\
&=  |\mathcal{C}_A \cup (\mathcal{C}_{A\cup B}\setminus \mathcal{C}_A) | + |\mathcal{C}|\\
&=  |\mathcal{C}_A| + |\mathcal{C}_{A\cup B}\setminus \mathcal{C}_A| + |\mathcal{C}| \\
&=  |\mathcal{C}_A| + |(\mathcal{C}_{A\cup B}\setminus \mathcal{C}_A) \cup \mathcal{C}| \\
&\leq \rho(A) + \rho(B),
\end{align*}
where the inequality at the end follows from the fact that \( (\mathcal{C}_{A\cup B}\setminus \mathcal{C}_A) \cup \mathcal{C} \) is an independent subset of \( B \) and hence can be extended to a basis of \( B \).
\end{proof}

Part (i) in the above proposition is called the \emph{monotonicity property}, and (ii) is known as the \emph{submodular inequality}. Using the latter, one easily obtains 
\begin{equation}  \label{rank difference}
\rho(B\setminus A) \geq \rho(B) - \rho(A) \quad \mbox{whenever } A \subseteq B \subseteq E \mbox{ and } B \mbox{ has finite rank}.
\end{equation}

An \emph{automorphism} of a finitary matroid \( (E, \mathcal{I}) \) is a bijective map \( f: E \rightarrow E \) such that for every \( A \subseteq E \), we have \( f(A) \in \mathcal{I} \) if and only if \( A \in \mathcal{I} \).

\begin{example} \label{universal ind group}
Let \( G \) be an arbitrary group.  Take \( E = G \) and \( \mathcal{I} = 2^G \), so that \( \mathcal{I} \) is the universal structure on \( E \).  Then for every \( g\in G \), the left-multiplication map \( x \mapsto gx \) from \( G \) to itself is an automorphism of \( (E,\mathcal{I}) \).
\end{example}

Our focus in this paper will be on situations where, as in the above example, the ground set of our finitary matroid is a group and the operation of the group is compatible with the independence structure.

\begin{definition} \label{IndGr}
An \emph{independence group} is a finitary matroid \( (G, \mathcal{I})\) where \( G \) is a group and for every \( g \in G \), the map 
\begin{equation*}
\begin{aligned}
G &\longrightarrow G \\
x &\longmapsto gx
\end{aligned}
\end{equation*}
is an automorphism of \( (G, \mathcal{I})\).
We also refer to \( G \) itself as an independence group under these circumstances.
\end{definition}

\begin{example}  \label{linear pi-group}
Let \( K \subseteq L \) be a field extension and let \( (E,\mathcal{I}) \) be the finitary matroid introduced in Example~\ref{linear pi}, here taking \( E = G = L^\times \). Then \( (G,\mathcal{I}) \) is an independence group:  if a subset \( A \) of \( L^{\times} \) is linearly independent over \( K \), then so is the set \( gA = \{ gx : x \in A \} \) (for any \( g \) in \( G \)). 
\end{example}

We remark that if \( G \) is an independence group and at least one of its elements is not a loop, then no element of \( G \) can be a loop; for if \( x,y \in G \) and \( x \) is not a loop, then \( \rho( \{ y \} ) = \rho( xy^{-1} \{ y \}) = \rho (\{ x \}) = 1 \).

\section{Matchings in Independence Groups}\label{MatchedSets}

We begin this section by giving, in two steps, a definition of matching for pairs of sets in an independence group. Our definition is motivated primarily by the notions of matching introduced in \cite{Losonczy 1} and \cite{Eliahou 2}.  We then describe two settings where matchings can be considered.  Other settings will be discussed in Section~\ref{ModSec}. At the end of this section, we briefly discuss a simple necessary condition for a pair of sets to be matched.

\begin{definition}  \label{matchable to def}
Let \( G \) be an independence group and let \( A \) be a finite-rank subset of \( G \).  Let \( \mathcal{A} \) be a basis of \( A \) and let \( \mathcal{B} \) be an independent subset of \( G \). We say that \( \mathcal{A} \) is \emph{matchable to} \( \mathcal{B} \) \emph{relative to} \( A \) if there exists a bijection \( f : \mathcal{A} \longrightarrow \mathcal{B} \) such that \( af(a) \notin A \) for all  \( a \in \mathcal{A} \).
\end{definition}

\begin{definition}  \label{matching def}
Let \( G\) and \(A\) be as in Definition~\ref{matchable to def}. Let \( B \) be a subset of \( G \) with \( \rho(B) = \rho(A) \).  We say that \( A \) is \emph{matched to} \( B \) if for every basis  \( \mathcal{A} \) of \( A \) and every basis \( \mathcal{B} \) of \( B \), we have that \( \mathcal{A} \) is matchable to \( \mathcal{B} \) relative to \( A \). We also say in this instance that the pair \( (A,B) \) is matched.
\end{definition}

Note that if \( A \) and \( B \) are subsets of an independence group and \( \rho(A) = \rho(B) = 0 \), then \(A \) is (vacuously) matched to \( B \).

\begin{example} \label{basic group example}
Let \( G \) be any independence group with universal structure \( \mathcal{I} \) and let \( A, B \) be finite subsets of \( G \) with \( |A| = |B| \).  Then \( A \) has a unique basis, namely itself, and likewise for \( B \).  Definition~\ref{matching def} therefore reduces to the following: \( A \) is matched to \( B \) provided there exists a bijection \( f: A \longrightarrow B\) such that \( af(a) \notin A \) for all \( a \in A \).  This agrees with the definition of matching given in \cite{Losonczy 1}.
\end{example}

Let \( K \subseteq L \) be a field extension, put \( G = L^{\times} \), and let \( \mathcal{I} \) be the collection of all subsets of \( G \) that are linearly independent over \( K \).  Recall from Example~\ref{linear pi-group} that \( G \) is an independence group when paired with \( \mathcal{I} \).  Now let \( A, B \) be $n$-dimensional $K$-subspaces of \( L \) with \( 0 < n < \infty \), and note that each of the sets \( A\setminus \{ 0 \} \) and \( B \setminus \{ 0 \} \) has rank \( n \) in \( G \). One can ask for conditions under which \( A\setminus \{ 0 \} \) is matched to \( B \setminus \{ 0 \} \).  A preliminary answer will be given in Section~\ref{transversal section}, and a more illuminating, structural one will be presented in Section~\ref{e-transform and char sec}.

\begin{remark}  \label{EL matching remark}
Let \( K,L,A, \) and \( B \) be as in the above paragraph.  A different but (as we will eventually show) equivalent conception of matching for this setting was introduced by Eliahou and Lecouvey in \cite{Eliahou 2}\footnote{We note that the authors of \cite{Eliahou 2} developed their theory in the more general setting of an extension of division rings \( K \subseteq L \), where \( K \) is contained in the center of \( L \).}.  More precisely, let \( \mathcal{A} = \{a_1, \ldots ,a_n\} \) and \( \mathcal{B} = \{ b_1, \ldots, b_n\} \) be $K$-bases of \( A \) and \( B \), respectively. In \cite{Eliahou 2}, \( \mathcal{A} \) is said to be matched to \( \mathcal{B} \) if  
\begin{equation} \label{EL condition}
a_i^{-1}A \cap B \;\subseteq\; \mbox{span}_K\big(\mathcal{B}\setminus \{ b_i \}\big) \quad \text{for all }i = 1,2,\ldots,n.
\end{equation}
The subspace \( A \) is then said to be matched to the subspace \( B \) if for every ordered $K$-basis \( \mathcal{A} \) of \( A \), there exists an ordered $K$-basis \( \mathcal{B} \) of \( B \) such that \( \mathcal{A} \) is matched to \( \mathcal{B} \).

Observe that whenever condition (\ref{EL condition}) is satisfied, one has \(a_i b_i\notin A\) and hence \(a_i b_i\notin \mathcal{A}\) for every \(i\). Hence the assignment \(a_i\mapsto b_i\) yields a matching, in the sense of Example~\ref{basic group example}, from \(\mathcal{A}\) to \(\mathcal{B}\) inside the multiplicative group \(L^\times\).  

Although condition (\ref{EL condition}) may initially seem rather different from the product-avoidance condition in Definition~\ref{matchable to def}, it will be shown in Section~\ref{transversal section} that \( A \) is matched to \( B \) in the sense of \cite{Eliahou 2} if and only if \( A\setminus \{ 0 \} \) is matched to \( B \setminus \{ 0 \} \) in the independence group \( G \).
\end{remark}

We conclude this short section with a basic observation about matchability.

\begin{proposition}  \label{necessary condition}
Let \( G \) be an independence group and suppose \( A,B \) are finite-rank subsets of \( G \) with \( \rho(A) = \rho(B) > 0 \). If \( 1 \in B \), then \( (A,B) \) is not matched.
\end{proposition}

\begin{proof}
Note that \( 1 \) is not a loop, since otherwise every element of \( G \) would be a loop and then \( A \) and \( B \) would not have positive rank.  Therefore, if \( 1 \in B \), we can extend \( \{ 1 \} \) to a basis \( \mathcal{B} \) of \( B \).  Then, appealing directly to Definition~\ref{matchable to def}, we find that no basis of \( A \) can be matchable to \( \mathcal{B} \) relative to \( A \).  The conclusion follows. 
\end{proof}

One may ask whether the converse of the above proposition is generally true.  It is not.  For example, in \cite{Losonczy 2} it was shown that in an abelian group \( G \) (with, in the terminology of this paper, a universal independence structure), the converse of the above holds if and only if \( G \) is torsion-free or of prime order.  We will address this issue in a far more comprehensive way in Section~\ref{self}.

\section{Independent Transversals and General Matching Criteria}  \label{transversal section}

Our primary goal in this section is to derive a general test for determining whether or not a pair \( (A,B) \) of sets in an independence group is matched.  Once established, it will be used to reconcile our definition of matching with the one discussed in Remark~\ref{EL matching remark}.  It will also be employed later, in Section~\ref{e-transform and char sec}, to prove a structural characterization of matchability. We begin with a brief discussion of the general theory of independent transversals in a finitary matroid.

Let \( E \) be a set and let \( \mathfrak{A} = (A_i : i \in I ) \) be a family of subsets \( A_i \) of \( E \). A \emph{transversal} of \( \mathfrak{A} \) is a subset \( X \) of \( E \) such that there exists a bijection \( \theta : X \rightarrow I \) with the property that \( x \in A_{\theta(x)} \) for all \( x \in X \). 

Now let \( (E, \mathcal{I}) \) be a finitary matroid and let \( \mathfrak{A} = (A_i : i \in I ) \) be a family of subsets of \( E \). An \emph{independent transversal} of \( \mathfrak{A} \) is a transversal \( X \) of \( \mathfrak{A} \) such that \( X \in \mathcal{I} \).

The result below, known as Rado's theorem, is of central importance in the study of independent transversals. A proof can be found in \cite{Mirsky}.  Note that the sets \( A_i \) in the statement are not required to be finite.  We remark that the theorem holds in the more general setting of a pre-independence space (where the hereditary and exchange axioms hold, but not necessarily the finite-character axiom).

\begin{theorem}[Rado] \label{Rado}
Let \( (E,\mathcal{I}) \) be a finitary matroid and let \( \mathfrak{A} = (A_i : i \in I ) \) be a family of subsets of \( E \), with \( I \) finite. Then \( \mathfrak{A} \) possesses an independent transversal if and only if, for every \( J \subseteq I \), 
\[ 
\rho\left(\bigcup_{i \in J}A_i\right) \;\geq\; |J|.
\]
\end{theorem}

We will use Rado's theorem here to obtain some useful general criteria for the matchability of independent sets and of pairs \( (A,B) \) in an independence group.

\begin{theorem}  \label{specialized}
Let \( G \) be an independence group and let \( n \) be a nonnegative integer.  Suppose that \( A \) is a subset of \( G \) of rank \( n \).   Let \( \mathcal{A} \) be a basis of \( A \) and let \( \mathcal{B} \subseteq G \) be an independent set with \( |\mathcal{B}| = n \). Then \( \mathcal{A} \) is matchable to \( \mathcal{B} \) relative to \( A \) if and only if, for every  \( J \subseteq \mathcal{A} \), 
\[
\rho\left(\bigcap_{a \in J} (a^{-1}A \cap \mathcal{B})\right) \;\leq\; n - |J|,
\]
where the intersection is understood to equal \( \mathcal{B} \) when \( J \) is empty.
\end{theorem}

\begin{proof}
If \( n = 0 \), the assertion holds vacuously, so assume \( n > 0 \). Suppose that \( \mathcal{A} \) is matchable to \( \mathcal{B} \) relative to \( A \), and let \( f: \mathcal{A} \rightarrow \mathcal{B} \) be a bijection such that \( af(a) \notin A \) for all \( a \in \mathcal{A} \).  Suppose \( J \subseteq \mathcal{A} \).  Observe that for each \( a \in \mathcal{A} \),
\[
a^{-1}A \cap \mathcal{B} \;\subseteq\; \mathcal{B}\setminus \{ f(a)\}.
\] 
Hence
\[
\bigcap_{a \in J}(a^{-1}A \cap \mathcal{B}) \;\subseteq\; \bigcap_{a \in J}\big(\mathcal{B}\setminus\{ f(a) \}\big) \;=\; \mathcal{B} \setminus \{f(a) : a\in J \}, 
\]
where we agree that the intersections equal \( \mathcal{B} \) if \( J \) is empty.  Using Proposition~\ref{rank properties}, the independence of \( \mathcal{B} \), and the fact that \( f \) is injective, we obtain 
\[
\rho\left( \bigcap_{a \in J}(a^{-1}A \cap \mathcal{B}) \right) \;\leq\; \rho\big(\mathcal{B} \setminus \{f(a) : a\in J \}\big) \;=\; n - |J|,
\]
as desired. 

Conversely, assume that the inequalities in the statement hold.  Define a family \( \mathfrak{D} = (D_a : a \in \mathcal{A}) \) of subsets of \( \mathcal{B} \) by setting \( D_a = \mathcal{B} \setminus a^{-1}A \). We will use Theorem~\ref{Rado} to show that \( \mathfrak{D} \) possesses an independent transversal.  Suppose \(J \subseteq \mathcal{A} \).  Observe that 
\[ 
\bigcup_{a \in J}D_a \;=\; \bigcup_{a\in J}(\mathcal{B}\setminus a^{-1}A)  \;=\; \mathcal{B} \mathbin{\big\backslash} \bigcap_{a\in J}(a^{-1}A \cap \mathcal{B}).
\]
Taking ranks and using inequality (\ref{rank difference}) in Section~\ref{Prelims}, we get
\begin{equation*}
\begin{aligned}
\rho\left( \bigcup_{a\in J}D_a\right) &\;=\; \rho\left( \mathcal{B} \mathbin{\big\backslash} \bigcap_{a\in J}(a^{-1}A \cap \mathcal{B})\right) \\
&\;\geq\; \rho(\mathcal{B}) - \rho\left( \bigcap_{a\in J}(a^{-1}A\cap \mathcal{B})\right) \\
&\;\geq\; n - (n- |J|) \;=\; |J|,
\end{aligned}
\end{equation*}
where the last inequality follows from our hypothesis.  We now apply Theorem~\ref{Rado}, obtaining an independent transversal \( X \) of \( \mathfrak{D} \) with \( X \subseteq \mathcal{B} \). By the definition of transversal, there exists a bijection \( \theta : X \rightarrow \mathcal{A} \) such that \( x \in D_{\theta(x)} \) for all \( x \in X \).  Clearly \( X = \mathcal{B} \).  Let \( f = \theta^{-1} \).  Then \( f \) is a bijection from \( \mathcal{A} \) to \( \mathcal{B} \) such that \( f(a) \notin a^{-1}A \) for all \( a \in \mathcal{A} \).  Equivalently, we have \( af(a) \notin A \) for all \( a \in \mathcal{A} \), so that \( \mathcal{A} \) is matchable to \( \mathcal{B} \) relative to \( A \).
\end{proof}

The next theorem provides a general tool for deciding whether a pair \( (A,B) \) is matched. It can be viewed as a generalization of the commutative case of Proposition 3.1 in \cite{Eliahou 2}.

\begin{theorem}  \label{existence condition}
Let \( G \) be an independence group and let \( n \) be a nonnegative integer.  Suppose \( A,B \) are subsets of \( G \), both of rank \( n \). Let \( \mathcal{A} \) be a basis of \( A\).  Then \( \mathcal{A} \) is matchable, relative to \( A \), to every basis of \( B \) if and only if, for every \( J \subseteq \mathcal{A} \), 
\[
\rho\left(\bigcap_{a \in J} (a^{-1}A \cap B)\right) \;\leq\; n - |J|,
\]
where the intersection is understood to equal \( B \) when \( J \) is empty.
\end{theorem}

\begin{proof}
The statement is true when \( n = 0 \), so suppose \( n > 0 \). Assume that the inequalities in the statement hold, and let \( \mathcal{B} \) be a basis of \( B \).  Observe that
\[
\bigcap_{a \in J} (a^{-1}A \cap \mathcal{B}) \;\subseteq\; \bigcap_{a \in J} (a^{-1}A \cap B) \quad \text{for all }J \subseteq \mathcal{A},
\]
where the intersections equal \( \mathcal{B} \) and \( B \), respectively, if \( J = \emptyset \). Hence, by our hypothesis and Proposition~\ref{rank properties}, 
\[
\rho\left(\bigcap_{a \in J} (a^{-1}A \cap \mathcal{B})\right) \;\leq \; n - |J| \quad \text{for all }J \subseteq \mathcal{A}.
\]
Applying Theorem~\ref{specialized}, we conclude that \( \mathcal{A}\) is matchable to \( \mathcal{B} \) relative to \( A \).

For the converse, assume that there exists a subset \( K \) of  \( \mathcal{A} \) such that
\[
\rho\left(\bigcap_{a \in K} (a^{-1}A \cap B)\right) \;>\; n - |K|.
\]
Let \( \mathcal{C} \) be a basis of \( \bigcap_{a \in K} (a^{-1}A \cap B) = \big( \bigcap_{a \in K} a^{-1}A \big) \cap B \) and choose a basis \( \mathcal{D} \) of \( B \). By Proposition~\ref{basis properties}, there exists a basis \( \mathcal{D}^* \) of \( B \) such that \( \mathcal{C} \subseteq \mathcal{D}^* \subseteq \mathcal{C} \cup \mathcal{D} \).  Notice that \( \mathcal{C} \) is a subset of \( \big( \bigcap_{a \in K} a^{-1}A \big) \cap \mathcal{D}^* = \bigcap_{a \in K} (a^{-1}A \cap \mathcal{D}^*) \), and so
\[
\rho\left(\bigcap_{a \in K} (a^{-1}A \cap \mathcal{D}^*)\right) \,\geq\, \rho\left(\bigcap_{a \in K} (a^{-1}A \cap B)\right) \,>\, n - |K|.
\]
Therefore, by Theorem~\ref{specialized}, \( \mathcal{A} \) is not matchable to \( \mathcal{D}^* \) relative to \( A \).  We conclude that \( A \) is not matched to \( B \).
\end{proof}

In Remark~\ref{EL matching remark} we described a notion of matching, originating in \cite{Eliahou 2}, for pairs of finite-dimensional subspaces in a field extension. We now show that the definition of matching in this paper, when restricted to that setting (but with the zero element removed), is equivalent to the one in \cite{Eliahou 2}.

\begin{corollary} \label{reconciliation corollary}
Let \( K \subseteq L \) be a field extension and let \( n \) be a nonnegative integer. Let \( A, B \) be $K$-subspaces of \( L \) with \( \dim_K A = \dim_K B = n \).  Let \( L^\times \) be the independence group defined in Example~\ref{linear pi-group}.  Put \( \widetilde{A}= A\setminus\{ 0 \} \) and \( \widetilde{B} = B \setminus \{ 0 \} \).  Then \( \widetilde{A} \) is matched to \( \widetilde{B} \) in \( L^\times \) if and only if \( A \) is matched to \( B \) in the sense of Remark~\ref{EL matching remark}.
\end{corollary}

\begin{proof}
By Proposition 3.1 in \cite{Eliahou 2}, \( A \) is matched to \( B \) in the sense of that paper if and only if, for every $K$-basis \( \mathcal{A} \) of \( A \) and every subset \( J \) of \( \mathcal{A} \), 
\begin{equation*}
\dim_K \left( \bigcap_{a \in J} (a^{-1}A \cap B) \right) \;\leq\;  n - |J|.
\end{equation*}
Fix a $K$-basis \( \mathcal{A} \) of \( A \), and observe that 
\[
a^{-1}\widetilde{A}\cap \widetilde{B} = (a^{-1}A\cap B)\setminus \{0 \} \quad \text{for all } a \in \mathcal{A}.
\]
Hence for each \( J \subseteq \mathcal{A} \),
\[
\bigcap_{a \in J} (a^{-1}\widetilde{A} \cap \widetilde{B}) \;=\; \left( \bigcap_{a \in J} (a^{-1}A \cap B) \right)\mathbin{\Big\backslash} \{ 0 \},
\]
where the intersections are understood to equal \( \widetilde{B} \) and \( B \), respectively, when \( J = \emptyset \).  It follows that for each \( J \subseteq \mathcal{A} \), 
\begin{equation*}
\rho \left( \bigcap_{a \in J} (a^{-1}\widetilde{A} \cap \widetilde{B}) \right) \;=\; \dim_K \left( \bigcap_{a \in J} (a^{-1}A \cap B) \right),
\end{equation*}
where on the left-hand side we are taking the rank in \( L^\times \). Also, note that a $K$-basis of \( A \) is the same thing as a basis of \( \widetilde{A} \) in \( L^\times \). Therefore, by Theorem~\ref{existence condition}, \( \widetilde{A} \) is matched to \( \widetilde{B} \) in \( L^\times \) if and only if \( A \) is matched to \( B \) in the sense of Remark~\ref{EL matching remark}.
\end{proof}

\section{Modularity and Examples}\label{ModSec}

In order to develop a structural theory for matchings in an independence group, we will find it helpful to work with flats and adopt the assumption of modularity (definitions given below).  This will enable us to generalize and make use of a tool from additive number theory. We begin by formalizing what it means for an element to depend on a subset in a finitary matroid.

Let \( (E,\mathcal{I}) \) be a finitary matroid, and let  \( x \in  E \) and \( A \subseteq E \).  Following \cite{Mirsky}, we write \( x \mid A \), and we say that \( x \) \emph{depends on} \( A \), if either \( x \in A \) or else \( x \notin A \) and \( \{ x \} \cup B \notin \mathcal{I} \) for some independent subset \( B \) of \( A \).  For \( A,B \subseteq E \), we write \( A\mid B \) provided \( x\mid B \) for all \( x \in A \).    

A proof of the following proposition can be found in \cite{Mirsky}, where finitary matroids are called independence spaces.

\begin{proposition}  \label{dependence} 
Let \( (E,\mathcal{I}) \) be a finitary matroid.  Let \( x \in E \) and \( A,B \subseteq E \). If \( x \mid A \) and \( A \mid B \), then \( x \mid B \).
\end{proposition}

The \emph{closure operator} for a finitary matroid \( (E,\mathcal{I}) \) is the map
\[ \mbox{cl}: 2^E \longrightarrow 2^E 
\] 
defined by 
\[
\mbox{cl}(A) = \{ x \in E : x \textup{ depends on } A \} \quad \text{for all }A \subseteq E. 
\]

The next proposition collects together some elementary properties of the closure operator. In the literature, proofs of these often rely on definitions and assumptions that differ somewhat from those used in this paper. For this reason, and for the sake of convenience, we provide a brief proof.

\begin{proposition} \label{closure properties}
Let \( (E, \mathcal{I}) \) be a finitary matroid.  Then the closure operator \( \textup{cl}: 2^E \rightarrow 2^E \) has the following properties:
\begin{itemize}
\item[(i)]  \( A \subseteq \textup{cl}(A) \) for all \( A \subseteq E \).
\item[(ii)]  \( \textup{cl}(A) \subseteq \textup{cl}(B) \) whenever \( A \subseteq B \subseteq E \).
\item[(iii)]  \( \textup{cl}\big(\textup{cl}(A)\big) = \textup{cl}(A) \) for all \( A \subseteq E \).
\item[(iv)] \( \rho\big(\textup{cl}(A)\big) = \rho(A) \) for all \( A \subseteq E \). 
\end{itemize}
\end{proposition}

\begin{proof}
Property (i) follows immediately from the definition of the closure operator.  For (ii), suppose \( A \subseteq B \subseteq E \) and \( x \in \textup{cl}(A) \). If \( x \in B \), we clearly have \( x \in \textup{cl}(B) \), so assume \( x \notin B \).  Then \( x \notin A \). The assumption \( x \in \textup{cl}(A) \) then implies \( \{ x \} \cup C \notin \mathcal{I} \) for some independent subset \( C \) of \(A \).  Since \( C \subseteq B \), it follows that \( x \mid B \).  Thus \( x \in \textup{cl}(B) \).

Turning to (iii), we first note that \( \textup{cl}(A) \subseteq \textup{cl}(\textup{cl}(A)) \) by property (i).  For the reverse inclusion, suppose \( y \in \textup{cl}(\textup{cl}(A)) \).  Then \( y \mid \textup{cl}(A) \). Since \( \textup{cl}(A) \mid A \), it now follows from Proposition~\ref{dependence} that \( y\mid A \), i.e., \( y \in \textup{cl}(A) \).

For (iv), choose a basis \( \mathcal{A} \) of \( A \) and assume for a contradiction that \( \mathcal{A} \) is not a basis of \( \textup{cl}(A) \). By (i), \( \mathcal{A} \) is an independent subset of \( \textup{cl}(A) \), so it can be extended to a basis \( \mathcal{A}^* \) of \( \textup{cl}(A) \).  Choose \( z \in \mathcal{A}^* \setminus \mathcal{A} \).  Then \( z \in \textup{cl}(A) \), so \( z \mid A \).  We also have \( A \mid \mathcal{A} \) (otherwise, there exists \( a \in A \setminus \mathcal{A} \) such that \( \{ a \} \cup \mathcal{A} \in \mathcal{I} \), contradicting the fact that \( \mathcal{A} \) is a maximal independent subset of \( A \)).  Hence \( z \mid \mathcal{A} \) by Proposition~\ref{dependence}. Since \( z \notin \mathcal{A} \), this implies there is an independent subset \( D \) of \( \mathcal{A} \) such that \( \{ z \} \cup D \notin \mathcal{I} \).  But \( \{ z \} \cup D \subseteq \mathcal{A}^* \), giving us a contradiction.
\end{proof}

In a finitary matroid \( (E,\mathcal{I}) \), a \emph{flat} is a subset \( A \) of \( E \) such that \( \textup{cl}(A) = A \). We will need the following well-known pair of facts about flats.  Once again, we provide a short proof.

\begin{proposition} \label{flat properties}
Let \( (E, \mathcal{I}) \) be a finitary matroid. Then the following statements hold:
\begin{itemize}
\item[(i)] If \( A,B \) are flats of \( E \), then so is \( A \cap B \).
\item[(ii)]  Suppose \(A \) is a finite-rank subset of \( E \).  Then \( A \) is a flat of \( E \) if and only if \( \rho\big(A \cup \{ x \}\big) = \rho(A) + 1 \) for all \( x \in E \setminus A \).
\end{itemize}
\end{proposition}

\begin{proof}
For (i), it suffices to show that \( \textup{cl}(A \cap B) \subseteq A \cap B \), in view of Proposition~\ref{closure properties}. Suppose for a contradiction that there is an element 
\[ 
x \in \textup{cl}(A \cap B)\setminus (A \cap B), 
\] 
and assume without loss of generality that \( x \notin A \).  Since \( x \mid A \cap B \), we have \( \{ x \} \cup X \notin \mathcal{I} \) for some independent subset \( X \) of \( A \cap B \).  Now, \( X \subseteq A \), so we also have \( x \mid A \).  But then, as \( A \) is a flat, it follows that \( x \in A \), which is absurd. 

Turning to (ii), we first assume that \( A \) is a flat.  Suppose \( y \in E \setminus A \). Then \( y \notin \textup{cl}(A) \), so that \( \{y \} \cup Y \in \mathcal{I} \) for all independent subsets \( Y \) of \( A \).  Choose a basis \( \mathcal{A} \) of \( A \). Then \( \{ y \} \cup \mathcal{A} \in \mathcal{I} \), and it is clear that \( \{ y \} \cup \mathcal{A} \) is a maximal independent subset (a basis) of \( \{ y \} \cup A \). Therefore 
\[ 
\rho(\{ y \} \cup A) = |\{ y \} \cup \mathcal{A}| = |\mathcal{A}| + 1 = \rho(A) + 1. 
\]

Finally, assume that \( A \) is not a flat, and let \( z \in \textup{cl}(A) \setminus A \).  Then \( \{ z \} \cup Z \notin \mathcal{I} \) for some independent subset \( Z \) of \( A \).  Extend \( Z \) to a basis \( \mathcal{Z} \) of \( A \), and then extend \( \mathcal{Z} \) to a basis \( \mathcal{Z}^* \) of \( \{ z \} \cup A \). We cannot have \( z\in  \mathcal{Z}^* \), since otherwise the subset \( \{ z \} \cup Z \) of  \( \mathcal{Z}^* \) would be independent, a contradiction.  Thus \( \mathcal{Z}^* \) is contained in \( A \) and so is a basis of \( A \). Hence
\[
\rho(\{ z \} \cup A) = |\mathcal{Z^*}| = \rho(A), 
\]  
completing the proof.
\end{proof}

Next we show that the rank and (non-)flatness of a set in an independence group are invariant under left translation.

\begin{proposition} \label{preservation of flatness and rank}
Let \( G \) be an independence group, let \( A \) be a subset of \( G \), and let \( g \in G \).  Then:
\begin{itemize}
\item[(i)] \(  \rho(gA) = \rho(A) \).
\item[(ii)] If \( A \) is a flat of \( G \), then so is \( gA \).
\end{itemize}
\end{proposition}

\begin{proof}
Let \( \mathcal{I} \) denote the independence structure on \( G \). 
For (i), we recall that by the definition of independence group, for any \( S \subseteq G \) and any \( h \in G \), we have \( hS \in \mathcal{I} \) if and only if \( S \in \mathcal{I} \), where \( hS = \{ hx : x \in S \} \). It immediately follows that \( \rho(gA) = \rho(A) \).

Assume now, toward a contradiction, that \( A \) is a flat and there is an element \( x \in \textup{cl}(gA)\setminus gA \).  Then \( \{ x \} \cup C \notin \mathcal{I} \) for some independent subset \( C \) of \( gA \).  Put \( D = g^{-1}C \).  Then \( D \) is an independent subset of \( A \) while \( \{g^{-1}x\} \cup D \notin \mathcal{I} \).  Since \( g^{-1}x \notin A \), this means that \( g^{-1}x \mid A \), i.e., \( g^{-1}x \in \textup{cl}(A) = A \), which is absurd.  Thus (ii) is proved. 
\end{proof}

\begin{definition} \label{modular definition}
A finitary matroid \( (E, \mathcal{I}) \) is called \emph{modular} if for all flats \(A,B\) of \( E \), 
\[
\rho(A) + \rho(B) \,=\, \rho(A \cup B) + \rho(A \cap B).
\]
A \emph{modular independence group} is an independence group which is modular as a finitary matroid.
\end{definition}

The rest of this section is devoted to presenting a variety of examples of modular independence groups.  These include groups with the universal independence structure, groups of the form \( L^\times / K^\times \), where \( K \subseteq L \) is a field extension and the independence structure is based on linear independence, and a class of examples constructed using coset partitions.

\begin{proposition} \label{basic modular examples}
The finitary matroids discussed in Examples~\ref{universal structure} and \ref{linear pi-group} are modular.  More specifically, the following statements hold:
\begin{itemize}
\item[(i)] Let \(E\) be a set equipped with the universal independence structure \(\mathcal{I}=2^E\). Then for every \(X\subseteq E\), we have
\[
\textup{cl}(X)=X \qquad\text{and} \qquad \rho(X)=|X|.
\]
Consequently, every subset of \(E\) is a flat, and \((E,\mathcal{I})\) is modular.

\item[(ii)] Let \(K\subseteq L\) be a field extension, let \(E=L^\times\), and let \(\mathcal{I}\) consist of the subsets of \(L^\times\) that are linearly independent over \(K\). Then for every \(X\subseteq L^\times\), we have
\begin{align*}
\textup{cl}(X) = (\textup{span}_K X)\setminus \{0\}
\quad \mbox{and} \quad \rho(X) = \dim_K\big(\textup{span}_K X\big).
\end{align*}
Consequently, the flats are precisely the sets of the form \( U \setminus\{0\} \), where \(U\) is a \(K\)-subspace of \(L\), and \((E,\mathcal{I})\) is modular.
\end{itemize}
\end{proposition}

\begin{proof}
We first consider the universal independence structure. Let \( X\subseteq E \) and suppose  \( x\in E \). By definition, \( x \mid X \) if and only if \( x \in X \) or else \( x \notin X \) and \( \{ x \} \cup C \notin \mathcal{I} \) for some independent subset \( C \) of \( X \).  Since every subset of \(E\) is independent, it follows that \( x \mid X  \,\mbox{ if and only if }\, x \in X \). In other words, \( \textup{cl}(X)=X \).  Moreover, as \(X\) itself is independent, it is the unique maximal independent subset of \(X\), giving us \( \rho(X)=|X| \).

Let \( A, B \subseteq E \). We have, in cardinal arithmetic, 
\[
|A|+|B| = |A\cup B|+|A\cap B| 
\]
(see, e.g., Exercise 8.2(9) in \cite{Pinter}), which translates to
\[
\rho(A)+\rho(B) = \rho(A\cup B)+\rho(A\cap B).
\]
Thus \( (E,\mathcal{I}) \) is modular.

For (ii), let \(X\subseteq L^\times\). We first verify that \( \textup{cl}(X) = (\textup{span}_K X) \setminus \{0\} \). Suppose \(x\in\textup{cl}(X)\). Then either \(x\in X\), in which case \( x \in \textup{span}_K X \), or else \( x \notin X \) and there exists \( C \in \mathcal{I} \) such that \( C \subseteq  X \) and \( \{ x \} \cup C \notin \mathcal{I} \), which yields 
\[ 
x \in \textup{span}_K C  \subseteq \textup{span}_K X. 
\] 
Noting that \( x \neq 0 \) since \( x \in L^\times \), we conclude that \( x \in (\textup{span}_K X)\setminus \{ 0 \} \).

For the reverse inclusion, suppose \( y \in L^{\times}\setminus \textup{cl}(X)\).  Then \( y\notin X \) and \( \{y\} \cup C \in \mathcal{I} \) for every \( C \in \mathcal{I} \) with \( C \subseteq X \).  This is equivalent to the condition that \( y \notin \textup{span}_K C \) for every \( C \subseteq X \) such that \( C \) is linearly independent over \( K \).  Hence \( y \notin \textup{span}_K X \), and we have proved that \( \textup{cl}(X) = (\textup{span}_K X) \setminus \{0\} \).  

From the above, we see that a subset \( F \) of \( L^\times \) is a flat if and only if \( F = (\textup{span}_K F) \setminus\{0\} \). In other words, the flats are precisely the sets of the form \( U\setminus \{0\} \), where \(U\) is a $K$-subspace of \(L\).

The equation  \( \rho(X) = \dim_K\big(\textup{span}_K  X \big) \) holds because a subset \( S \) of \( X \) is a basis in the finitary matroid sense if and only if \( S \) is a $K$-basis of \( \textup{span}_K X \).

Finally, we verify the modularity property. For any flats 
\( A = U\setminus \{0\} \) and \( B = V\setminus \{0\} \) in \( L^\times \), we have
\[
A\cap B
=
(U\cap V)\setminus\{0\} \qquad \text{and} \qquad \textup{span}_K (A\cup B) = U+V.
\]
Therefore
\[
\rho(A\cap B)=\dim_K(U\cap V) \qquad \text{and} \qquad \rho(A\cup B)=\dim_K(U+V).
\]
Irrespective of the cardinals involved, we have (see, e.g., Chapter 1 of \cite{Roman}) 
\[
\dim_K U+\dim_K V = \dim_K(U+V)+\dim_K(U\cap V).
\]
Thus
\[
\rho(A)+\rho(B) = \rho(A\cup B)+\rho(A\cap B),
\]
and we conclude that \( (L^\times, \mathcal{I}) \) is modular.
\end{proof}

Next we describe the projective independence groups arising from field extensions.  Note that we do not restrict our attention to extensions of finite degree until Proposition~\ref{projective submonoids}, where we establish a one-to-one correspondence between the intermediate fields of the extension and the flat submonoids of the independence group. 

A general reference for projective spaces in the matroid context is \cite{Oxley}.

\begin{definition}  \label{proj equiv def}
Let \( K \subseteq L \) be a field extension and let 
\[
\pi \colon L^\times \longrightarrow L^\times/K^\times
\]
be the quotient map. A subset \(A \) of \( L^\times / K^\times \) is called \emph{projectively independent} if for one---and hence every---choice of representatives \(x_a \in a \) (for each \( a \in A\)), the set \( \{x_a : a \in A \} \) is linearly independent over \(K\). For a $K$-subspace \(W \) of  \(L\), we write
\[
\mathbb{P}_K(W)=\pi\bigl(W\setminus\{0\}\bigr).
\]
\end{definition}

\begin{proposition}[Projective independence groups]  \label{projective independence group}
Let \(K\subsetneq L\) be a field extension and let
\[
G=L^\times/K^\times.
\]
Define \( \mathcal{I} \) to be the collection of all projectively independent subsets of \( L^\times / K^\times \).  Then \( \mathcal{I} \) is an independence structure on \( G \).  With this structure, \( G \) is a modular abelian independence group.

Now define, for each \( A \subseteq G \),
\[
V_A \,=\, \operatorname{span}_K \{x\in L^\times:\pi(x)\in A\} \,\subseteq\, L.
\]
Then
\[
\textup{cl}(A)=\mathbb P_K(V_A) \qquad \text{and} \qquad  \rho(A)=\dim_K V_A.
\]
Consequently, the flats of \( G \) are precisely the sets \( \mathbb P_K(W) \),  where \(W\) ranges over the $K$-subspaces of \(L\).
\end{proposition}

\begin{proof}
Using the fact that the definition of projective independence does not depend on the choice of representatives, it is straightforward to show using Example~\ref{linear pi} that the collection \( \mathcal{I} \) is an independence structure. 

To see that \( \mathcal{I} \) is compatible with the group operation of \( G \), let \( g \in G \) and \( A \subseteq G \).  Choose a representative \( x_a \in a \) for each \(a \in A \), and a representative \( z \in g \).  Then \( \{zx_a : a \in A \} \) is a set of representatives for the elements of \( gA \).  Clearly \( \{zx_a : a \in A \} \) is linearly independent over \( K \) if and only if \( \{ x_a : a \in A \} \) is. From this we get \( gA \in \mathcal{I} \) if and only if \( A \in \mathcal{I} \). Thus \( G \) is an independence group. It is plainly abelian.

Next we verify the assertion about \( \textup{cl}(A) \).  Suppose \( yK^\times \in \textup{cl}(A) \) with \( y \in L^\times \).  If \( yK^\times \in A \), then by the definition of \( V_A \) we have \( y \in V_A \), and so \( yK^\times = \pi(y) \in \pi\big(V_A\setminus \{ 0 \}\big) \).  Now suppose \( yK^\times \notin A \).  Then \( \{ yK^\times \} \cup C \notin \mathcal{I} \) for some independent subset \( C \) of \( A \).  Choose a representative \( y_c \in c \) for each \(c \in C\).  Then \( \{ y_c : c \in C \} \) is linearly independent over \( K \) while \( \{ y \} \cup \{ y_c : c \in C \} \) is linearly dependent. Hence
\[
y \in \big(\text{span}_K\{ y_c : c \in C \}\big) \setminus \{ 0 \} \,\subseteq\, V_A\setminus \{ 0 \},  
\]
so that \( yK^\times = \pi(y) \in \mathbb P_K(V_A) \).

For the reverse inclusion, suppose \( wK^\times \in  \mathbb P_K(V_A) \) with \( w \in V_A\setminus \{ 0 \} \). If \( wK^\times \in A \), there is nothing to do, so assume \( wK^\times \notin A \). Then there exists a subset \( D \) of \( L^\times \) such that
\begin{itemize}
\item \( \pi(d) \in A \) for all \( d \in D \),
\item \( D \) is linearly independent over \( K \),
\item \( w \in (\text{span}_K D)\setminus D \).
\end{itemize}
Thus \( \pi(D) \) is a subset of \( A \) such that \( \pi(D) \in \mathcal{I} \) while \( \{ wK^\times \} \cup \pi(D) \notin \mathcal{I} \). This gives us \( wK^\times \in \textup{cl}(A) \), as desired.

From the above, together with the easily verifiable fact that \( V_{\mathbb{P}_K(W)} = W \) for any $K$-subspace \( W \) of \( L \), one obtains that the flats of \( G \) are precisely the sets \( \mathbb{P}_K(W) \), where \( W \) ranges over the $K$-subspaces of \( L \).

To see that \( \rho(A) = \dim_K V_A \), let \( \mathcal{A} \) be a basis of \( A \) and choose a representative \( z_a \in a \) for each \( a \in \mathcal{A}\). The set \( \{ z_a : a \in \mathcal{A} \} \) is linearly independent over \( K \), and we claim that it is a $K$-basis of \( V_A \).  Assume it is not.  Let  \( \mathcal{B} \) be a $K$-basis of \( V_A \) such that \( \mathcal{B} \subseteq \{ x \in L^\times : \pi(x) \in A \} \). Then there exists a $K$-basis \( \mathcal{B}^* \) of \( V_A \) such that 
\[
\{ z_a : a \in \mathcal{A}\} \,\subsetneq\, \mathcal{B}^* \,\subseteq\, \{z_a : a \in \mathcal{A}\} \cup \mathcal{B}.
\]
We now have that \( \pi(\mathcal{B}^*) \) is a projectively independent subset of \( A \) and 
\[ 
\mathcal{A} \,=\, \pi\big(\{z_a:a \in \mathcal{A}\}\big) \,\subsetneq\, \pi(\mathcal{B}^*), 
\] 
contradicting the fact that \( \mathcal{A} \) is a basis of \( A \). The claim is proved, and it follows that 
\[
\rho(A) \,=\, |\mathcal{A}| \,=\, |\{ z_a : a \in \mathcal{A} \}| \,=\, \dim_K V_A.
\]

Finally, if \(U,W \) are \(K\)-subspaces of \( L \), it is easy to see that 
\[
\mathbb P_K(U)\cap\mathbb P_K(W) = \mathbb P_K(U\cap W) \quad \text{and} \quad V_{\mathbb{P}_K(U) \cup \mathbb{P}_K(W)} = U + W.
\]
Using these equations along with the dimension formula 
\[
\dim_K U+\dim_K W = \dim_K(U+W)+\dim_K(U\cap W),
\]
we find that 
\[
\rho\big(\mathbb{P}_K(U)\big) + \rho\big(\mathbb{P}_K(W)\big) = \rho\big( \mathbb{P}_K(U) \cup \mathbb{P}_K(W) \big) + \rho\big(\mathbb{P}_K(U)\cap \mathbb{P}_K(W)\big),
\]
and modularity is proved.
\end{proof}

Let \( G \) be an independence group and let \( M \) be a submonoid of \( G \).  If \( M \) is also a flat of \( G \), we call it a \emph{flat submonoid}.

\begin{proposition}[Projective flat submonoids]  \label{projective submonoids}
Let \(K\subsetneq L\) be a finite extension of fields, and let
\[
L^\times/K^\times
\]
be the projective independence group of Proposition~\ref{projective independence group}. Then the assignment
\[
F\longmapsto F^\times/K^\times
\]
is a bijection from the set of intermediate fields
\[
K\subseteq F\subseteq L
\]
to the set of flat submonoids of \(L^\times/K^\times\).

In particular, if \(M\subseteq L^\times/K^\times\) is a nonempty flat, then \(M\) is a submonoid if and only if there is a unique intermediate field \(F\) such that
\[
M=\mathbb P_K(F)=F^\times/K^\times.
\]
In this case,
\[
\rho(M)=[F:K].
\]
Thus every flat submonoid of \( L^\times/K^\times \) is a subgroup.
\end{proposition}

\begin{proof}
Let \(F\) be an intermediate field of \( K \subsetneq L \). Since \(F\) is a $K$-subspace of \(L\), Proposition~\ref{projective independence group} shows that \(\mathbb P_K(F)=F^\times/K^\times \) is a nonempty flat of the projective independence group. Moreover, \(F^\times/K^\times\) is a subgroup, and hence a submonoid, of \(L^\times/K^\times\). Using the rank formula in Proposition~\ref{projective independence group}, we get 
\[
\rho\big(\mathbb{P}_K(F)\big) = \dim_K V_{\mathbb{P}_K(F)} = \dim_K F = [F:K].
\]

Conversely, let \(M\subseteq L^\times/K^\times\) be a flat submonoid. By Proposition~\ref{projective independence group}, there exists a $K$-subspace \(W\subseteq L\) such that \( M = \mathbb P_K(W) \). Since \(M\) is a submonoid, it contains the identity element of \( L^\times / K^\times \).  Thus, if \( \pi : L^\times \rightarrow L^\times / K^\times \) is the quotient map, we have \(\pi(1)\in\mathbb P_K(W) = \pi\big(W\setminus \{0\}\big) \), so \(W\) contains a nonzero element of \(K\). Because \(W\) is a $K$-subspace, it follows that \( K\subseteq W \).

We claim that \(W\) is closed under multiplication. Let \(x,y\in W\). There is nothing to prove if \(x=0\) or \(y=0\), so assume \(x,y\neq 0\). Then
\[
\pi(x),\pi(y)\in M.
\]
Since \(M\) is a submonoid,
\[
\pi(xy)=\pi(x)\pi(y)\in M=\mathbb P_K(W).
\]
Consequently, there exist \(w\in W\setminus \{ 0 \} \) and \(c\in K^\times\) such that \( xy = cw \). Because \(W\) is a $K$-subspace, this implies \(xy\in W\). Therefore
\(W\) is a \(K\)-subalgebra of \(L\).

It remains to show that \(W\) is closed under inverses. Let \( x\in W \setminus \{ 0 \} \).  Multiplication by \(x\) defines a $K$-linear map
\begin{equation*}
\begin{aligned}
m_x :\; &W \longrightarrow W \\
&w \longmapsto xw.
\end{aligned}
\end{equation*}
This map is injective because \(W\subseteq L\) and \(L\) is a field. Moreover, \(W\) is finite-dimensional over \(K\), since \( W\subseteq L \) and \( [L:K] < \infty \). Thus \(m_x\) is surjective. Since \(1\in W\), there exists \(y\in W\)
such that
\[
m_x(y) = xy = 1.
\]
Therefore \(x^{-1}=y\in W\). It follows that \(W\) is a field. Setting
\(F=W\), we obtain an intermediate field
\[
K\subseteq F\subseteq L
\]
such that
\[
M=\mathbb P_K(F)=F^\times/K^\times.
\]

To see that \(F\) is uniquely determined by \(M\), let \( F' \) be an intermediate field of \( K \subsetneq L \) with \( M = \mathbb{P}_K(F') \). Then in the notation of Proposition~\ref{projective independence group},
\[ 
F' = V_{\mathbb{P}_K(F')} = V_M = V_{\mathbb{P}_K(F)} = F,
\]
completing the proof.
\end{proof}

In the proposition below, we describe a different way to construct modular independence groups.  Here the groups arise as partition matroids, with blocks chosen to be the left cosets of some fixed subgroup.

\begin{proposition}[Coset-partition independence groups] \label{coset partition proposition}

Let \(G\) be a group and let \(H\) be a subgroup. Define \(\mathcal{I}_H\) to be the collection of all subsets \( A \)  of \( G \) satisfying
\[
|A \cap gH|\leq 1 \qquad\text{for all } g\in G.
\]
Then \( \mathcal{I}_H \) is an independence structure on \( G \).  With this structure, \( G \) is a modular independence group.

Moreover, for every \(X\subseteq G\),
\[
\textup{cl}(X)=XH \qquad \text{and} \qquad \rho(X)=|\{xH:x\in X\}|.
\]
Consequently, the flats of \(G\) are precisely the unions of left cosets of \(H\).
\end{proposition}

\begin{proof}
It is clear that the hereditary axiom holds for \( \mathcal{I}_H \). Suppose that \(A,B \in \mathcal{I}_H\) are finite and \( |B| = |A|+1 \).  Since an independent set contains at most one element from each left coset of \(H\), the set \(B\) meets \(|B|\) distinct left cosets, whereas \(A\) meets only \(|A|\) distinct left cosets. Hence there exists \(b\in B\) such that \( bH\cap A = \emptyset \). It follows that \( A \cup\{b\}\in\mathcal{I}_H \), so the exchange axiom holds.

Now let \(A \subseteq G\) and suppose that every finite subset of \(A\) belongs to \(\mathcal{I}_H\). If two distinct elements \(x,y\in A\) belonged to the same left coset of \(H\), then the finite set \(\{x,y\}\) would be dependent. Therefore \(A \in\mathcal{I}_H\), and the finite-character axiom holds for \( \mathcal{I}_H \).

For every \(g\in G\), left multiplication by \(g\) permutes the left cosets of \(H\), since \( g(xH)=(gx)H \). It follows that a subset of \(G\) is independent if and only if its left translate by \(g\) is independent. Thus \( G \) is an independence group when equipped with \( \mathcal{I}_H \).

Let \(X\subseteq G\). A maximal independent subset of \(X\) contains exactly one representative from every left coset of \(H\) that meets \(X\). Indeed, it contains at most one such representative by independence, and maximality forces it to contain at least one.  Therefore
\[
\rho(X)=|\{xH:x\in X\}|.
\]

We next verify the formula for \( \textup{cl}(X) \). Let \(y\in XH\). If \(y\in X\), then \(y\in\textup{cl}(X)\) by definition. Otherwise, there exists \(x\in X\) such that \( x \neq y \) and \( yH = xH \).  The singleton \(\{x\}\) is independent, while \(\{x,y\}\) is dependent. Hence \(y\) depends on \(X\), and so \( y\in\textup{cl}(X) \). 

Conversely, suppose \(z \notin XH\). Then the coset \(zH\) does not meet \(X\). Consequently, for every independent subset \( C \) of  \( X \), the set \( C \cup \{z\} \) is independent. Hence \(z\) does not depend on \(X\), and therefore \( z \notin\textup{cl}(X) \). We conclude that \( \textup{cl}(X)=XH \). It follows that \(X\) is a flat if and only if \(X=XH\), which is
equivalent to saying that \(X\) is a union of left cosets of \(H\).

Finally, we verify modularity. For each \( X \subseteq G \), define 
\[
\Omega_X = \{gH: g \in G \text{ and }gH\subseteq X\}.
\]
Let \(F\) and \(K\) be flats of \(G\), and observe that, since each of these is a union of left cosets of \( H \), we have 
\[
\rho(F)=|\Omega_F|,
\qquad
\rho(K)=|\Omega_K|,
\]
and
\[
\Omega_{F\cup K}=\Omega_F\cup\Omega_K,
\qquad
\Omega_{F\cap K}=\Omega_F\cap\Omega_K.
\]
Therefore, in cardinal arithmetic,
\begin{align*}
\rho(F)+\rho(K)
&=
|\Omega_F|+|\Omega_K|\\
&=
|\Omega_F\cup\Omega_K|
+
|\Omega_F\cap\Omega_K|\\
&=
\rho(F\cup K)+\rho(F\cap K).
\end{align*}
Thus \(G\) is modular.
\end{proof}

Below we present an example of a coset-partition independence group, constructed by taking \( G \) and \( H \) to be the edge space and cycle space, respectively, of a graph \( \Gamma \).  An introductory discussion of these spaces can be found in \cite{Diestel}.

\begin{example}[A cycle-space independence group]
\label{cycle space independence example}
Let \(\Gamma=(\mathsf{V},\mathsf{E})\) be a finite simple graph, with vertex set \( \mathsf{V} \) and edge set \( \mathsf{E} \), and consider the Boolean group
\[
G=(2^{\mathsf{E}},\triangle),
\]
where \(\triangle\) denotes symmetric difference. Let
\[
\mathcal{Z}(\Gamma)
=
\{Z\subseteq \mathsf{E} :
\text{every vertex has even degree in }(\mathsf{V},Z)\}.
\]
Then \(\mathcal{Z}(\Gamma)\) is a subgroup of \(G\).  It is known as the \emph{cycle space} of \(\Gamma\).

Define an independence structure \( \mathcal{I} = \mathcal{I}_{\mathcal{Z}(\Gamma)} \) on \(G\) by declaring a subset \( \mathcal{A} \) of \(2^{\mathsf{E}}\) to be independent if no two distinct members of \( \mathcal{A} \) belong to the same coset of \(\mathcal{Z}(\Gamma)\). Equivalently, \( \mathcal{A} \in \mathcal{I} \) provided
\[
X\triangle Y\notin\mathcal{Z}(\Gamma)
\qquad
\text{for all distinct }X,Y\in \mathcal{A}.
\]

By Proposition~\ref{coset partition proposition}, this structure makes \( G \) a modular independence group. The rank of any subset \( \mathcal{A} \) of \( 2^{\mathsf{E}} \) is given by
\[
\rho(\mathcal{A}) = \left| \left\{ X\triangle\mathcal{Z}(\Gamma): X\in \mathcal{A} \right\} \right|,
\]
where
\[
X\triangle\mathcal{Z}(\Gamma)
=
\{X\triangle Z:Z\in\mathcal{Z}(\Gamma)\},
\]
and its closure is
\[
\textup{cl}(\mathcal{A}) = \bigcup_{X\in \mathcal{A}} \bigl(X\triangle\mathcal{Z}(\Gamma)\bigr).
\]
Thus the flats are precisely the unions of cosets of the cycle space.
\end{example}

\begin{remark}
The preceding example of an independence group, constructed relative to the cycle-space cosets in the edge space of a graph, should not be confused with the usual graphic matroid: here the ground set is the collection \(2^\mathsf{E}\) of all edge subsets of \(\Gamma\), rather than the edge set \(\mathsf{E}\) itself.
\end{remark}

\section{Generalized e-transform and Structural Characterization of Matched Pairs} \label{e-transform and char sec}

In this section we develop a structural theory for matchings in modular abelian independence groups.  Our first task will be to introduce a generalized version of the $e$-transform from additive number theory (Theorem~\ref{Dyson G}). We will then use this tool to prove our main characterization result, Theorem~\ref{matching characterization}. The generalized $e$-transform will also be employed later, in the proof of Theorem~\ref{threshold product growth}. Several applications of Theorem~\ref{matching characterization} (and an alternative formulation, Theorem~\ref{simpler characterization}) will be given, starting in this section and continuing into Section~\ref{self}. 

We begin with a simple lemma that explains how multiplication and the closure operator interact.

\begin{lemma} \label{product and closure}
Let \( G \) be an independence group.  Then the following assertions hold:
\begin{itemize}
\item[(i)] Suppose \( x \in G \) and \( B \subseteq G\) with \( x \mid B \).  Then \( gx \mid gB \) for all \( g\in G \). 
\item[(ii)]  Let \( A,B \) be subsets of \( G \). Then \( A\big(\textup{cl}(B)\big) \subseteq \textup{cl}(AB) \).
\end{itemize}
\end{lemma}

\begin{proof}
Let \( \mathcal{I} \) denote the independence structure on \( G \). For (i), the conclusion is clear if \( x \in B \), so assume \( x \notin B \).  Then \( \{ x \} \cup C \notin \mathcal{I} \) for some independent subset \( C \) of \( B \).   Hence, for any \( g \in G \), we have that \( \{gx \} \cup gC \notin \mathcal{I} \) and \( gC \) is an independent subset of \( gB \).  Since \( gx \notin gB \), we conclude that \( gx \mid gB \). 

Turning to (ii), let us suppose \( a \in A \) and \( x \in \textup{cl}(B) \). Then \( x \mid B \), and part (i) gives us \( ax \mid aB \), i.e., \( ax \in \textup{cl}(aB)\). Since \(\textup{cl}(aB) \subseteq \textup{cl}(AB) \), it follows that \( ax \in \textup{cl}(AB) \), as desired.
\end{proof}

We are ready to introduce the generalized $e$-transform. Roughly speaking, the transform works in the following way.  Given a pair \( (S,R) \) of nonempty finite-rank flats, the transform enlarges \( S \) and shrinks \( R \), while preserving the sum of the ranks of the two flats.  Iteration of this construction terminates at a pair \( (S',R') \), where the first flat is invariant under multiplication by the second. Details are given in the proof of the theorem below.  Note that this result underlies all of the structural theory that follows.

\begin{theorem} \label{Dyson G}
Let \( G \) be a modular abelian independence group. Suppose that \( A, S, R \) are nonempty flats of \( G \), all three of finite rank. Assume that \( 1 \in R \) and \( SR \subseteq A \). Then there exist nonempty flats \( S', R' \) of \( G \), both of finite rank, such that the following four conditions hold:
\begin{itemize}
\item[(i)]  \( S \subseteq S'R' \subseteq \textup{cl}(SR) \subseteq A \),
\item[(ii)]  \( 1 \in R' \subseteq R \),
\item[(iii)] \( \rho(S') + \rho(R') = \rho(S) + \rho(R) \),
\item[(iv)] \( S'R' = S' \subseteq A \).
\end{itemize}
\end{theorem}

\begin{proof}
First observe that if \( SR = S \), we can simply take \( S' = S \) and \( R' = R \).  Assume that  \( SR \neq S \). Then, since \( 1 \in R \), there exists an element \( x_0 \in SR\setminus S \).  Write \( x_0 = er \) with \(e \in S \) and \( r \in R \). Define
\[
S_1 = \textup{cl}(S \cup eR)\quad \text{and} \quad R_1 = R \cap (Se^{-1}), 
\]
and note that \( S_1, R_1 \) are flats, by Propositions~\ref{closure properties}, \ref{flat properties}, and \ref{preservation of flatness and rank}.  We call the pair \( (S_1,R_1) \) the $e$-transform of \( (S,R) \). We claim that the five conditions below hold:
\begin{itemize}
\item[(a)] \( S_1R_1 \subseteq \textup{cl}(SR) \subseteq A \),
\item[(b)] \( 1 \in R_1 \subseteq R \),
\item[(c)] \( \rho(S_1) + \rho(R_1) = \rho(S) + \rho(R) \),
\item[(d)] \( S \subseteq S_1 \subseteq A \mbox{ and }\rho(S) < \rho(S_1) \),
\item[(e)] \( S_1, R_1 \) are nonempty and have finite rank.
\end{itemize}
Condition (b) is clear. For (a), we first use Lemma~\ref{product and closure} and commutativity to obtain 
\begin{equation*}
\begin{aligned}
S_1R_1 \,=\, \textup{cl}(S\cup eR)R_1 &\,=\, R_1\big( \textup{cl}(S\cup eR)\big) \\
&\,\subseteq\, \textup{cl}\big( R_1(S\cup eR) \big) \\
&\,=\, \textup{cl}\big( (S\cup eR)R_1 \big) . 
\end{aligned}
\end{equation*}

Since \( (S\cup eR)R_1  \subseteq SR \subseteq A \), we also have 
\[
\textup{cl}\big( (S \cup eR)R_1 \big) \,\subseteq\, \textup{cl}(SR) \,\subseteq\, \textup{cl}(A) \,=\, A.
\]
Combining inclusions gives us what we want.  

The inclusion \( S \subseteq S_1 \) in (d) is immediate.  The other inclusion there follows from (a) and (b) (specifically, \( S_1R_1 \subseteq A \) and \( 1 \in R_1 \)).  The inequality in (d) is a consequence of Proposition~\ref{flat properties}, since \( S \) is a finite-rank flat, \( S \subseteq S_1\), and \( x_0 \in S_1 \setminus S \). Observe that (e) follows easily from (b) and (d). 

Finally, for (c) we first note that \( S \) and  \( eR \) are flats. Then, using Propositions~\ref{closure properties} and \ref{preservation of flatness and rank} along with commutativity and the modularity of \( G \), we compute  
\begin{align*}
\rho(S_1) \,=\, \rho\big(\textup{cl}(S \cup eR)\big) &\,=\, \rho(S \cup eR) \\
&\,=\, \rho(S) + \rho(eR) - \rho(S \cap eR) \\
&\,=\, \rho(S) + \rho(R) - \rho\big( e^{-1}(S\cap eR)\big) \\
&\,=\, \rho(S) + \rho(R) - \rho(e^{-1}S\cap R) \\
&\,=\, \rho(S) + \rho(R) - \rho(R_1).
\end{align*}
All five conditions have been verified.

If \( S_1R_1 \neq  S_1 \), the above is repeated, this time by forming the $e$-transform of \( (S_1,R_1) \) rather than \( (S,R) \).  The procedure continues until we reach  a pair \( (S_n,R_n) \) of nonempty flats, both of finite rank, satisfying \( S_nR_n =  S_n \). This is guaranteed to occur because the sequence \( \rho(S_1), \rho(S_2), \ldots \) is strictly increasing and bounded above by the finite cardinal \( \rho(A) \). Let \( S' = S_n \) and \( R' = R_n \). These sets clearly satisfy the conditions in the statement. 
\end{proof}

\begin{remark}
The assumptions of commutativity and modularity play distinct roles in the proof of Theorem~\ref{Dyson G}. Commutativity is used to ensure that the transformed product \( S_1R_1 \) remains inside the closure of the original product \( SR \), while modularity gives
$$
\rho(S_1)+\rho(R_1)=\rho(S)+\rho(R).
$$
Without commutativity, the required product containment may fail, and without modularity, one has only
$$
\rho(S_1)+\rho(R_1)\leq\rho(S)+\rho(R),
$$
whereas in the proof of Theorem~\ref{matching characterization} below, one needs the reverse inequality.
\end{remark}

The following theorem gives a structural characterization of the (nontrivial) pairs of flats that are matched.  Using this theorem and its alternative version, Theorem~\ref{simpler characterization}, we will obtain a result on the matchability of a flat with itself (Theorem~\ref{self matching criterion}) and a characterization of a certain global matching property (Theorem~\ref{MMP}).

We note that the structural aspect of the theorem comes from the condition \( SR = S \), the implications of which will be seen later in Lemma~\ref{monoid lemma} and in subsequent results.

\begin{theorem}  \label{matching characterization}
Let \( G \) be a modular abelian independence group and let \( n \) be a natural number.  Suppose  \( A,B \) are flats of \( G \), both of rank \( n \). Then \( A \) is matched to \( B \) if and only if, for every pair of nonempty flats \( S \subseteq A \) and \( R \subseteq \textup{cl}(B\cup \{ 1 \}) \) with \( SR = S \), we have 
\[
\rho(S) + \rho(R\cap B) \,\leq\, n.
\]
\end{theorem}

\begin{proof}
Assume that \( A \) is matched to \( B \).  Suppose that \( S \) and \( R \) are nonempty flats such that \( S \subseteq A \), \( R \subseteq \textup{cl}(B\cup \{ 1 \}) \), and \( SR = S \).  Choose a basis \( \mathcal{S} \) of \( S \) and, using Proposition~\ref{basis properties}, extend it to a basis \( \mathcal{A} \) of \( A \).  Applying Theorem~\ref{existence condition} to \( \mathcal{A} \) and taking \( J = \mathcal{S} \) in that theorem, we get
\begin{equation} \label{easy direction}
\rho\left(\bigcap_{a \in \mathcal{S}} (a^{-1}A \cap B)\right) \,\leq\, n - |\mathcal{S}| \,=\, n - \rho(S).
\end{equation}
Now, since \( SR = S \), we have \( aR \subseteq S \) for all \( a \in S \), and hence \( R \subseteq a^{-1}S \subseteq a^{-1}A \) for all \( a \in S \). Therefore 
\[
R\cap B \,\subseteq\, \bigcap_{a \in \mathcal{S}} (a^{-1}A \cap B).
\]
Combining this with inequality (\ref{easy direction}), we obtain 
\[
\rho(R\cap B) \,\leq\, n - \rho(S),
\]
as desired.

For the converse, suppose that \( A \) is not matched to \( B \).  Then, by Theorem~\ref{existence condition}, there is a basis \( \mathcal{A} \) of \( A \) and a nonempty subset \( J \) of \( \mathcal{A} \) such that 
\[
\rho\left(\bigcap_{a \in J} (a^{-1}A \cap B)\right) \,>\, n - |J|.
\]
Define
\[
S = \textup{cl}(J) \qquad \text{and} \qquad R = \bigcap_{a \in J} (a^{-1}A \cap B). 
\]
Note that \( R \) is nonempty, since it has positive rank.  It is also a flat, by Propositions~\ref{flat properties} and \ref{preservation of flatness and rank}. Observe also that \(S \) and \(R\) have finite rank and satisfy \( S \subseteq A \) and \( R \subseteq B \), as well as
\[
\rho(S) + \rho(R) \,>\, n.
\]
Define a flat \( R_0 \) by 
\[
R_0 = \textup{cl}(R \cup \{ 1 \}).
\]
Clearly \( R_0 \subseteq \textup{cl}(B \cup \{ 1 \})\). We claim that \( SR_0 \subseteq A \).  To see this, note first that by the definition of \( R \), the inclusion \( JR \subseteq A \) holds. Hence, by Lemma~\ref{product and closure} and commutativity,
\[
SR \,=\, \textup{cl}(J)R \,\subseteq\, \textup{cl}(JR) \,\subseteq\, \textup{cl}(A) \,=\, A.
\]
It follows that \( S(R \cup\{ 1 \}) \subseteq A \), and so, again by Lemma~\ref{product and closure},
\[
SR_0 \,=\, S\big(\textup{cl}(R\cup \{1\})\big) \,\subseteq\, \textup{cl}\big(S(R \cup \{1\})\big) \,\subseteq\, \textup{cl}(A) \,=\, A,
\]
establishing the claim.

We are now able to apply Theorem~\ref{Dyson G} to \( A, S, R_0 \). We obtain nonempty flats \( S', R_0' \) of finite rank such that 
\begin{itemize}
\item[(a)] \( 1 \in R_0' \subseteq R_0 \subseteq \textup{cl}(B \cup \{ 1 \}) \),
\item[(b)] \( \rho(S') + \rho(R_0') = \rho(S) + \rho(R_0) \),
\item[(c)] \( S'R_0' = S' \subseteq A \).
\end{itemize}

The proof will be complete if we can show that \( \rho(S') + \rho(R_0' \cap B ) > n \) (for then the pair \( (S',R_0') \) will serve as a counterexample to the property in the statement of the theorem).  We consider two cases.

{\bf Case 1:} \( 1 \in B \).  In this situation, \( \textup{cl}(B \cup \{ 1 \}) = B \), and so \( R_0' \cap B = R_0' \).  Hence
\begin{align*}
\rho(S') + \rho(R_0' \cap B) &\,=\, \rho(S') + \rho(R_0') \\
&\,=\, \rho(S) + \rho(R_0) \\
&\,\geq\, \rho(S) + \rho(R) \,>\, n,
\end{align*}
as desired.

{\bf Case 2:} \( 1 \notin B \).  By Propositions~\ref{closure properties} and \ref{flat properties}, we have \( \rho\big(\textup{cl}(B\cup \{ 1 \})\big) = \rho(B\cup \{ 1 \}) = n+1 \).  From this it follows that \( \rho(R_0' \cup B ) = n+1 \).  Indeed, the inclusion \( B \cup \{ 1 \} \subseteq R_0' \cup B \) implies \( n+1 = \rho(B \cup \{ 1 \}) \leq \rho(R_0' \cup B) \), while the containment \( R_0' \cup B \subseteq \textup{cl}(B \cup \{ 1 \}) \) gives us \( \rho(R_0' \cup B) \leq \rho\big(\textup{cl}(B\cup \{1\})\big) = n+1 \).   

Therefore, by modularity,
\begin{align*}
\rho(R_0'\cap B) &\,=\, \rho(R_0') + \rho(B) - \rho(R_0'\cup B) \\
&\,=\, \rho(R_0') + n - (n+1) \\
&\,=\, \rho(R_0') - 1.
\end{align*}
We also have
\[
\rho(R_0) = \rho\big(\textup{cl}(R \cup \{ 1 \})\big) = \rho(R \cup \{ 1 \}) =  \rho(R) + 1,
\]
again by Propositions~\ref{closure properties} and \ref{flat properties}.  Combining the last two calculations with (b) above, we obtain
\begin{align*}
\rho(S') + \rho(R_0'\cap B) &\,=\, \rho(S') + \rho(R_0') - 1 \\ 
&\,=\, \rho(S) + \rho(R_0) -1 \\
&\,=\, \rho(S) + (\rho(R) +1) -1 \,>\, n,
\end{align*}
which completes the proof.
\end{proof}

Next we show how the two main results of \cite{Aliabadi 5} can be recovered from Theorem~\ref{matching characterization} as special cases. 

Suppose first that \( G \) is an abelian group.  Equip \( G \) with the universal independence structure \( \mathcal{I} \) and recall from Proposition~\ref{basic modular examples} that \( G \) is then a modular independence group.  For this structure, every subset \( A \) of \( G \) is a flat and its rank is \( |A| \). Thus the inequality in Theorem~\ref{matching characterization} can be rewritten as
\[
|S| + |R \cap B| \,\leq\, n.
\]
Since \( |B| = n \), this is equivalent to 
\[
|S| \,\leq\, |B\setminus R|.
\] 
We therefore have the following corollary, a result which originally appeared as Theorem 2.1 in \cite{Aliabadi 5}.

\begin{corollary} \label{recover group char}
Let \( G \) be an abelian group (with universal independence structure) and let \( A,B \) be finite nonempty subsets of \( G \) with \( |A| = |B| \).  Then \( A \) is matched to \( B \) if and only if, for all nonempty subsets \( S \subseteq A \) and \( R \subseteq B\cup \{ 1 \} \) with \( SR = S \), we have 
\[
|S|  \;\leq\; |B\setminus R|.
\]
\end{corollary}

Now suppose that \( G \) is an abelian independence group of the type discussed in Example~\ref{linear pi-group}.  Thus \( G = L^{\times} \), where \( L \) is an extension of some field \( K \), and a subset of \( G \) is independent if and only if it is linearly independent over \( K \). It was shown in Proposition~\ref{basic modular examples} that \( G \) is modular and its flats are all of the form \( U\setminus \{ 0 \} \), where \( U \) is a $K$-subspace of \( L \). In addition, we have \( \rho\big(U\setminus \{ 0 \}\big) = \dim_K U \), and the closure of any subset \( X \) of \( G \) is \( (\textup{span}_K X)\setminus \{ 0 \} \).  

Assume here that \( K \neq L \).  Let \( A,B\) be $n$-dimensional $K$-subspaces of \( L \) with \( 0 < n < \infty \).  Suppose that \(S \subseteq A \) and \( R \subseteq B+K \) are nonzero $K$-subspaces. Set \( \widetilde{A} = A \setminus \{ 0 \} \), and likewise for \( B \), \( S \), and \( R \).  Then the condition \( \widetilde{S}\widetilde{R} = \widetilde{S} \) in \( G \) is equivalent to \( SR = S \) in \( L \).  The inequality 
\[
\rho(\widetilde{S}) + \rho(\widetilde{R} \cap \widetilde{B}) \,\leq\, n 
\]
is equivalent to 
\[
\dim_K S + \dim_K (R \cap B) \,\leq\, n, 
\]
which in turn can be rewritten as 
\[
\dim_K S \,\leq\, \dim_K \big(B/(R \cap B)\big). 
\]
Note also that if \( 1 \notin B \), then the inclusion \( \widetilde {R} \subseteq \textup{cl}\big(\widetilde{B} \cup \{ 1 \}\big) \) is equivalent to \( R \subseteq B \oplus K \). 

In view of the above and Corollary~\ref{reconciliation corollary}, we have the following corollary of Theorem~\ref{matching characterization}. This result is a reformulation of Theorem 3.1 in \cite{Aliabadi 5}.

\begin{corollary}  \label{recover linear char}
Let \( K \subsetneq L \) be a field extension and let \( A,B \) be $n$-dimensional $K$-subspaces of \( L \), with \( 0 < n < \infty \) and \( 1 \notin B \). Then \( A \) is matched to \( B \) in the sense of Remark~\ref{EL matching remark} if and only if, for every pair of nonzero $K$-subspaces \( S \subseteq A \) and \( R \subseteq B \oplus K \) with \( SR = S \), we have 
\[
\dim_K S \,\leq\, \dim_K \big(B/(R \cap B)\big). 
\] 
\end{corollary}

We make use of the following lemma in the proof of Theorem~\ref{simpler characterization} below, and we will need it again in Section~\ref{self}.

\begin{lemma} \label{SR=S lemma}
Let \( G \) be an abelian independence group and let \( S, R \) be nonempty subsets of \( G \).  Assume that \( S \) is a finite-rank flat and \( SR \subseteq S \).  Then \( SR = S \).
\end{lemma}

\begin{proof}
Let \( x \in R \).  By commutativity and Proposition~\ref{preservation of flatness and rank}, the set \( Sx \) is a flat with \( \rho(Sx) = \rho(S) \). By hypothesis, \( Sx \subseteq S \). Therefore, using finiteness of rank and Proposition~\ref{flat properties}, we obtain \( Sx = S \).  It follows that
\[
SR = \bigcup_{r \in R}Sr = \bigcup_{r\in R}S = S.
\]
\end{proof}

Below we present an alternative version of Theorem~\ref{matching characterization}. Here the condition involving \( S \) and \( R \) no longer involves the group identity element.  Note, however, that for some purposes it is convenient to work with Theorem~\ref{matching characterization}, e.g., when retrieving some previously known results.

\begin{theorem}  \label{simpler characterization}
Let \( G \) be a modular abelian independence group and let \( n \) be a natural number.  Suppose  \( A,B \) are flats of \( G \), both of rank \( n \). Then \( A \) is matched to \( B \) if and only if, for every pair of nonempty flats \( S \subseteq A \) and \( R \subseteq B \) with \( SR = S \), we have 
\[
\rho(S) + \rho(R) \,\leq\, n.
\]
\end{theorem}

\begin{proof}
First observe that if the condition involving \( S \) and \( R \) in Theorem~\ref{matching characterization} holds, then so does the one in the statement above. 

Suppose that the condition in Theorem~\ref{matching characterization} does not hold, and let \( S\subseteq A \) and \( R \subseteq \textup{cl}(B\cup\{1\})\) be nonempty flats such that \( SR = S \) and 
\begin{equation} \label{from other theorem}
\rho(S) + \rho(R\cap B) \,>\, n.
\end{equation}
Notice that, in order for the above inequality to hold, the flat \( R \cap B \) must be nonempty. Denote this flat by \( R' \).  Clearly \( R'\subseteq B \) and 
\[ 
SR' \,\subseteq\, SR \,=\, S. 
\]  
In fact, the last inclusion and Lemma~\ref{SR=S lemma} give us \( SR' = S \).

Thus, on account of inequality~(\ref{from other theorem}), the pair \( (S,R') \) violates the condition in the statement.
\end{proof}

\section{Self-matching, Submonoid Thresholds, and Product Growth} \label{self}

In this section we explore some consequences of Theorems~\ref{Dyson G} and \ref{matching characterization}. Our main results include a theorem on self-matching (Theorem~\ref{self matching criterion}), a theorem identifying the modular abelian independence groups having a certain global matching property (Theorem~\ref{MMP}), and a lower bound for the rank of \( XY \), where \( X \) and \( Y \) are nonempty finite-rank sets (Theorem~\ref{threshold product growth}). The lower bound involves a parameter \( \mu(G) \) defined in this section, called the submonoid-rank threshold of \( G \).

We also prove a result on the matchability of flats having rank less than \( \mu(G) \) (Corollary~\ref{low rank matching}) and another on the existence of unmatchable pairs with rank equaling \( \mu(G) \) (Proposition~\ref{sharp low rank matching}). Examples from Section~\ref{ModSec} are revisited in the context of the concepts and theorems presented in this section.

\begin{definition} \label{self-match def}
We say that a finite-rank subset \(A\) of an independence group is \emph{self-matched} if the pair \( (A,A) \) is matched.
\end{definition}

\begin{theorem}[Self-matching criterion]
\label{self matching criterion}
Let \( G \) be a modular abelian independence group, and let \(A\) be a flat of finite rank. Then \(A\) is self-matched if and only if \( 1\notin A \).
\end{theorem}

\begin{proof}
If \( 1 \in A \), then Proposition~\ref{necessary condition} immediately implies that \( A \) is not self-matched.

Assume \(1\notin A\). The empty set is self-matched, so assume \( A \neq \emptyset \). Suppose \(S,R\subseteq A\) are nonempty flats satisfying \( SR=S \).
We will show that
\[
\rho(S)+\rho(R)\leq\rho(A).
\]

We claim that \( S \cap R=\emptyset \).  Assume the contrary, and let  \(x\in S\cap R\).  Then the equation \( SR=S \) gives us \( Sx\subseteq S \). From this inclusion we get \( Sx = S \), by finiteness of rank along with commutativity and Propositions~\ref{flat properties} and \ref{preservation of flatness and rank}. Now, \(x\in S=Sx\), so there exists \(s\in S\) such that \( sx=x \). It follows that \(s=1\), and hence \( 1\in S\subseteq A \),
contrary to the assumption \(1\notin A\). The claim is established.

Since \( G \) is modular,
\[
\rho(S)+\rho(R)
=
\rho(S\cup R)+\rho(S\cap R).
\]
As \( S\cap R=\emptyset \), we have \( \rho(S\cap R)=0 \). Furthermore, \( S\cup R\subseteq A \), so the monotonicity property of rank gives 
\[
\rho(S)+\rho(R) = \rho(S\cup R) \leq \rho(A).
\]
Theorem~\ref{simpler characterization}, applied with \(B=A\), now shows that \(A\) is self-matched.
\end{proof}

Let \( G \) be a group and let \(R \subseteq G \).  We write \( \langle R \rangle \) for the submonoid of \( G \) generated by \( R \).  This is the intersection of all of the submonoids of \( G \) that contain \( R \). Note that \( \langle R \rangle \) consists of the identity element of \( G \) together with all products of the form \( x_1 \cdots x_k \), where \( k \) is a natural number and each \( x _i \) belongs to \( R \).

In what follows, when we say that a submonoid \( M \) of an independence group \( G \) has finite rank, we mean that \( \rho(M) \) is finite.

\begin{lemma} \label{monoid lemma}
Let \( G \) be an independence group and let \(S, R\) be nonempty subsets of \( G \) with \( SR \subseteq S \). Then the following statements hold:
\begin{itemize}
\item[(i)]  We have \( S\langle R \rangle \subseteq S \). 
\item[(ii)]  Suppose \( S \) has finite rank, say \( \rho(S) = n \).  Then \( \rho\big(\langle R \rangle\big) \leq n\).
\end{itemize}
\end{lemma}

\begin{proof}
Given \( a \in S \) and \( x \in \langle R \rangle \), we clearly have \( ax \in S \) if \( x = 1 \).  Suppose \( x = x_1x_2 \cdots x_k \), where \( k > 0 \) and each \( x_i \in R \). Then \( ax_1 \in SR \subseteq S \), hence \(ax_1x_2 = (ax_1)x_2 \in SR \subseteq S \), and so on, eventually yielding \( ax \in S \).  Thus (i) holds.

For (ii), let \( a' \) be an arbitrary element of \( S \). Using (i) and Proposition~\ref{preservation of flatness and rank}, we obtain \( \rho\big(\langle R \rangle\big) = \rho\big(a'\langle R \rangle\big) \leq \rho(S) \leq n\).
\end{proof}

\begin{definition} \label{matching property def}
Let \( G \) be an independence group. We say that \( G \) has the \emph{matching property} if for every pair \( (A,B) \) of finite-rank flats of \( G \) such that \( \rho(A) = \rho(B)\) and \( 1 \notin B \), we have that \(A\) is matched to \( B \).
\end{definition}

Below, the matching property is characterized for modular abelian independence groups by the absence of a submonoid satisfying a pair of rank conditions.

\begin{theorem}[Submonoids and the matching property] \label{MMP}
Let \( G \) be a modular abelian independence group.  Then \( G \) has the matching property if and only if \( G \) has no submonoid \( H \) such that 
\[ 
1 < \rho(H) < \rho(G)  \qquad \text{and} \qquad  \rho(H) < \infty. 
\]
\end{theorem}

\begin{proof}
Assume that \( G \) has no submonoid \( H \) satisfying the above rank conditions.  Let \( A,B \) be finite-rank flats in \( G \) with \( \rho(A) = \rho(B) \) and \( 1 \notin B \). Let \( n \) be the common rank of \( A \) and \( B \).  If \( n =0 \), then \( A \) is (vacuously) matched to \( B \), so assume \( n > 0 \).  We will use Theorem~\ref{matching characterization} to show that \( A \) is matched to \( B \).

Before proceeding, note that \( G \) has no loops; otherwise, every element of \( G \) would be a loop and then both \( A \) and \( B \) would have rank zero, contrary to assumption. 

Suppose we are given nonempty flats \( S \subseteq A \) and \(R \subseteq \textup{cl}\big(B\cup \{ 1 \}\big) \) with \( SR = S \).  Consider the submonoid \( \langle R \rangle \) of \( G \).  According to Lemma~\ref{monoid lemma}, the rank of this submonoid is no greater than \( n \). Observe also that \( n < \rho(G) \); indeed, since \( B \) is a flat and \( 1 \notin B \), we have
\[ \rho(G) \geq \rho(B \cup \{ 1 \}) = \rho(B) + 1 = n+1,
\] 
using Proposition~\ref{flat properties}.
By our hypothesis, then, we must have \( \rho\big(\langle R \rangle\big) \leq 1 \). We cannot have \( \rho\big(\langle R \rangle\big) = 0 \), since otherwise \( 1 \) would be a loop.  Thus \(\rho\big(\langle R \rangle\big) = 1 \). 

We claim that \( R \cap B = \emptyset \).  Assume the contrary and let \( x \in R\cap B \).  Then \( x, 1 \in \langle R \rangle \), and \( \{ x \} \) is a basis of \( \langle R \rangle \) (because \( G \) has no loops).  Hence \( 1~\mid~\{ x \} \).  We also have \( x \in B \), so that \( \{ x \} \mid B \).  It now follows from Proposition~\ref{dependence} that \( 1 \mid B \). Since \( B \) is a flat, this gives us \( 1 \in B \), a contradiction.

Using the claim, we now compute
\[
\rho(S) + \rho(R \cap B ) \,=\, \rho(S) + 0 \,\leq\, \rho(A) \,=\, n.
\]
Applying Theorem~\ref{matching characterization}, we conclude that \( A \) is matched to \( B \).

For the converse, assume that \( G \) has a submonoid \( H \) such that 
\[ 
1 < \rho(H) < \rho(G)  \qquad \text{and} \qquad  \rho(H) < \infty. 
\]
Let \( m = \rho(H) \).  Since \( m > 0 \), there are no loops in \( G \).  Extend \( \{ 1 \} \) to a basis \( \{1,b_1,\ldots,b_{m-1}\} \) of \( H \) and note that this is also a basis of \( \textup{cl}(H) \).  Choose \( x \in G \setminus \textup{cl}(H) \).  Then the set \( \{x,1,b_1,\ldots,b_{m-1} \} \) is independent.  Define
\[
A = \textup{cl}(H) \qquad \text{and} \qquad B = \textup{cl}\big(\{x,b_1,\ldots ,b_{m-1} \}\big).
\]
We have \( \rho(A) = \rho(B) = m \) and \( 1 \notin B \). We will use Theorem~\ref{matching characterization} to show that \( A \) is not matched to \( B \).  Take \( S = A \) and \( R = \textup{cl}\big(\{b_1,\ldots,b_{m-1} \}\big) \).  Using commutativity, two applications of Lemma~\ref{product and closure}, and the fact that \( H \) is a monoid, we obtain
\begin{align*}
SR \,=\, A\Big(\textup{cl}\big(\{b_1,\ldots,b_{m-1}\}\big)\Big)  &\,\subseteq\, \textup{cl}\big(A\{b_1,\ldots,b_{m-1}\}\big) \\
&\,=\, \textup{cl}\big(\textup{cl}(H)\{b_1,\ldots,b_{m-1}\}\big)  \\  
&\,\subseteq\,  \textup{cl}\Big(\textup{cl}\big(H\{b_1,\ldots,b_{m-1}\}\big)\Big) \\
&\,\subseteq\, \textup{cl}(H) = A = S.
\end{align*}
We have shown that \( SR \subseteq S \), and this in fact gives us \( SR = S\), by Lemma~\ref{SR=S lemma}.  Finally, observe that 
\[
\rho(S) + \rho(R\cap B) \,=\, \rho(A) + \rho(R) \,=\, m + (m - 1) \,>\, m.
\]
Applying Theorem~\ref{matching characterization} completes the proof.
\end{proof}

Motivated by Theorem~\ref{MMP}, we make the following definition.

\begin{definition} \label{submonoid rank threshold}
Let \(G\) be an independence group. Define
\[
\mu(G) =
\min\Bigl(
\bigl\{
\rho(H): H \text{ is a submonoid of }G\text{ and }
1<\rho(H)<\infty
\bigr\}
\cup\{\infty\}
\Bigr).
\]
We call \(\mu(G)\) the \emph{submonoid-rank threshold} of \(G\).  
\end{definition}

Note that \(\mu(G)=\infty\) precisely when every finite-rank submonoid of \(G\) has rank at most one.

\begin{corollary}[Low rank matching]
\label{low rank matching}
Let \(G\) be a modular abelian independence group and let \(A,B\) be finite-rank flats of \( G \) satisfying
\[
\rho(A)=\rho(B)=n<\mu(G)
\qquad\text{and}\qquad
1\notin B.
\]
Then \(A\) is matched to \(B\).
\end{corollary}

\begin{proof}
If \(n=0\), the assertion holds vacuously, so assume \(n>0\). We will apply Theorem~\ref{matching characterization}. Let \(S\subseteq A\) and \( R\subseteq\textup{cl}(B\cup\{1\}) \) be nonempty flats satisfying \( SR=S \). Set \( H=\langle R\rangle \). Since \(SR\subseteq S\), Lemma~\ref{monoid lemma} gives
\[
\rho(H)\leq\rho(S)\leq n<\mu(G).
\]
Note that \( G \) has no loops because \( n > 0 \). Since \(1\in H\), it follows that \(\rho(H)\geq1\). The definition of \(\mu(G)\) now forces \( \rho(H)=1 \).

Observe that \(R\cap B=\emptyset\); otherwise, we can argue as in the proof of Theorem~\ref{MMP} to obtain \( 1 \in B \), which is a contradiction.  It follows that
\[
\rho(S)+\rho(R\cap B)
=
\rho(S)
\leq
\rho(A)
=
n.
\]
Theorem~\ref{matching characterization} now shows that \(A\) is matched to \(B\).
\end{proof}

The following lemma will be used in the proofs of the next three results.

\begin{lemma} \label{flat submonoid closure}
Let \(G\) be an abelian independence group and let \(H \) be a submonoid of \(G\). Then \( M = \textup{cl}(H) \) is a flat submonoid of \(G\) with \( \rho(M) = \rho(H) \).
\end{lemma}

\begin{proof}
We have \(1\in M\) since \(1\in H \) and \( H \subseteq M\).  By Lemma~\ref{product and closure} and the equation \(HH = H\),
\[
HM \,=\, H\big(\textup{cl}(H)\big) \,\subseteq\,
\textup{cl}(HH) \,=\, \textup{cl}(H) \,=\, M.
\]
Since \(G\) is abelian, this also gives \(MH\subseteq M\). Applying Lemma~\ref{product and closure} once more, we obtain
\[
MM \,=\,
M\big(\textup{cl}(H)\big)
\,\subseteq\,
\textup{cl}(MH)
\,\subseteq\,
\textup{cl}(M)
\,=\, M.
\]
Thus \(M\) is a submonoid. It is clearly a flat.  Finally, Proposition~\ref{closure properties} gives
\[
\rho(M)=\rho\bigl(\textup{cl}(H)\bigr)=\rho(H).
\]
\end{proof}

The proposition below concerns the existence of unmatched pairs \( (A,B) \).

\begin{proposition}[Sharpness at the threshold] \label{sharp low rank matching}
Let \(G\) be a modular abelian independence group, and suppose that 
\[
1< n=\mu(G)<\rho(G),
\]
where \( n \) is a natural number. Then there exist finite-rank flats \(A,B\) of \( G \) such that
\[
\rho(A)=\rho(B)=n,\qquad 1\notin B,
\]
and \(A\) is not matched to \(B\).
\end{proposition}

\begin{proof}
Since \(n=\mu(G)\), there exists a submonoid \(H\) of \(G\) with \( \rho(H)=n \). Let \( A = \textup{cl}(H) \).  By Lemma~\ref{flat submonoid closure}, \( A \) is a flat submonoid of \( G \) such that \( \rho(A)=n \).

Note that \( G \) has no loops, since \( n > 0 \).  Extend \(\{1\}\) to a basis \( \{1,b_1,\ldots,b_{n-1}\} \) of \(A\). Since \(n<\rho(G)\), we can choose \(x\in G\setminus A\). Define
\[
B=\operatorname{cl}\bigl(\{x,b_1,\ldots,b_{n-1}\}\bigr).
\]
The set \( \{1,x,b_1,\ldots,b_{n-1}\} \) is independent and so \( \rho(B)=n \) and \( 1\notin B\).

Now let
\[
R=\operatorname{cl}\bigl(\{b_1,\ldots,b_{n-1}\}\bigr).
\]
Then \(R\subseteq B\) and \(\rho(R)=n-1\). Since \(A\) is a flat
submonoid,
\[
AR \,\subseteq\, AA \,=\, A.
\]
Lemma~\ref{SR=S lemma} now gives \( AR=A \). Finally,
\[
\rho(A)+\rho(R) \,=\, n+(n-1) \,>\, n.
\]
Thus \(A\) and \(B\) violate the condition in Theorem~\ref{simpler characterization}, and therefore \(A\) is not matched to \(B\).
\end{proof}

\begin{corollary}[Projective threshold and matching property] \label{projective threshold and matching property}
Let \(K\subsetneq L\) be a finite extension of fields, and let
\[
G = L^\times/K^\times
\]
be the projective independence group of Proposition~\ref{projective independence group}. Then
\[
\mu(G)
=
\min\bigl\{
[F:K]:
K\subsetneq F\subseteq L
\text{ is an intermediate field}
\bigr\}.
\]
Moreover, \(G\) has the matching property if and only if the extension \(K \subsetneq L\) has no nontrivial proper intermediate field.
\end{corollary}

\begin{proof}
By Proposition~\ref{projective independence group}, we have \( \rho(G)=[L:K]<\infty \). Let \(H\) be a submonoid of \(G\) with \(\rho(H)>1\). Since
\(H\subseteq G\), it has finite rank. By Lemma~\ref{flat submonoid closure}, \( M = \textup{cl}(H) \) is a flat submonoid of \(G\) satisfying \( \rho(M)=\rho(H) \). Proposition~\ref{projective submonoids} therefore gives a unique intermediate field \(F\), with \(K\subsetneq F\subseteq L\), such that
\[
M=F^\times/K^\times
\qquad\text{and}\qquad
\rho(H)=\rho(M)=[F:K].
\]

Conversely, every intermediate field \(F\) satisfying \(K\subsetneq F\subseteq L\) determines, by Proposition~\ref{projective submonoids}, a flat submonoid
\( F^\times/K^\times \) of rank \([F:K]>1\).  Definition~\ref{submonoid rank threshold} now gives
\[
\mu(G)
=
\min\bigl\{
[F:K]:
K\subsetneq F\subseteq L
\text{ is an intermediate field}
\bigr\}.
\]

Finally, Theorem~\ref{MMP} shows that \(G\) has the matching property if and only if there is no submonoid \(H\) of \(G\) satisfying
\[
1<\rho(H)<\rho(G)=[L:K].
\]
By the correspondence established above, such a submonoid exists if and only if there is an intermediate field \(F\) satisfying \( K\subsetneq F\subsetneq L\). This proves the final assertion.
\end{proof}

\begin{example}  \label{projective C21 example}
We examine the matching behavior of two different independence structures on the same group.  Consider the finite field extension
\[
K \,=\, \mathbb F_4 \,\subseteq\, L \,=\, \mathbb F_{64},
\]
and note that \([L:K]=3\).  Let \( G = L^\times/K^\times \).  Since the multiplicative group of a finite field is cyclic,
\[
G \,=\, \mathbb F_{64}^{\times}/\mathbb F_4^{\times} \,\cong\, C_{21}.
\]

We first equip \(G\) with the projective independence structure associated with \(\mathbb F_4\subseteq\mathbb F_{64}\). Since the extension has prime degree, it has no proper intermediate field. Therefore, by Corollary~\ref{projective threshold and matching property}, \(G\) has the matching property with this structure.

Now we equip the same abstract group \(G\cong C_{21}\) with the universal independence structure \( \mathcal{I} \). Recall that, relative to \( \mathcal{I} \), we have \( \rho(X)=|X|\) for all \(X\subseteq G\). The group \(C_{21}\) has a subgroup \(U\) of order \(3\); this is a submonoid satisfying
\[
1 \,<\, \rho(U) \,=\, 3 \,<\, 21 \,=\, \rho(G).
\]
Therefore, by Theorem~\ref{MMP}, \( G \) equipped with \( \mathcal{I} \) does not have the matching property.

Thus two independence structures on the same abstract group can produce different matching behavior.
\end{example}

We will use the following lemma in the proof of Theorem~\ref{threshold product growth}.

\begin{lemma} \label{product of finite rank}
Let \( G \) be an abelian independence group, and let \( X \) and \( Y \) be nonempty finite-rank subsets of \( G \). Then the product set \( XY \) has finite rank.
\end{lemma}

\begin{proof}
Let $\mathcal{X}$ and $\mathcal{Y}$ be bases of $X$ and $Y$, respectively.  Observe that both \( \mathcal{X} \) and \( \mathcal{Y} \) are finite, since $X$ and $Y$ have finite rank, and
\begin{equation*}
X\subseteq \operatorname{cl}(\mathcal{X})
\qquad\text{and}\qquad
Y\subseteq \operatorname{cl}(\mathcal{Y}).
\end{equation*}
Using commutativity, Proposition~\ref{closure properties}, and Lemma~\ref{product and closure}, we obtain
\[
\textup{cl}(\mathcal{X})\textup{cl}(\mathcal{Y})
\,\subseteq\,
\textup{cl}\big(\textup{cl}(\mathcal{X})\mathcal{Y}\big)
\,=\,
\textup{cl}\Bigl(\mathcal{Y}\big(\textup{cl}(\mathcal{X})\big)\Bigr)
\,\subseteq\,
\textup{cl}\big(\textup{cl}(\mathcal{Y}\mathcal{X})\big)
\,=\,\textup{cl}(\mathcal{X}\mathcal{Y}).
\]
Consequently,
\[
XY \,\subseteq\, 
\textup{cl}(\mathcal{X})\textup{cl}(\mathcal{Y}) \,\subseteq\, \textup{cl}(\mathcal{X}\mathcal{Y}).
\]
Since the product set \( \mathcal{X}\mathcal{Y}\)  is finite, we conclude that
\[
\rho(XY)
\,\leq\, \rho\big(\textup{cl}(\mathcal{X}\mathcal{Y})\big)
\,=\, \rho(\mathcal{X}\mathcal{Y})
\,\leq\, |\mathcal{X}\mathcal{Y}|
\, < \, \infty.
\]
\end{proof}

\begin{remark}
The assumption of commutativity in Lemma~\ref{product of finite rank} cannot be dropped.  To see this, let \( G \) be the free group on the symbols \( a \) and \(b\), and give \( G \) the coset-partition independence structure (as described in Proposition~\ref{coset partition proposition}) associated with the subgroup \( H \) generated by \( a \).  Then each of the sets \( X = H \) and \(Y = bH \) has rank $1$, while the product
\begin{equation*}
XY = \bigcup_{n \in \mathbb{Z}}a^nbH
\end{equation*}
meets infinitely many left cosets of \( H \) and hence has infinite rank. 
\end{remark}

We are ready to present our main result on product growth for subsets of a modular abelian independence group.  An important ingredient in the proof will be the generalized $e$-transform (Theorem~\ref{Dyson G}).

\begin{theorem}  \label{threshold product growth}
Let \(G\) be a modular abelian independence group and let \(X,Y\) be nonempty finite-rank subsets of \( G \).  Then:
\begin{itemize}
\item[(i)] We have
\[ \rho(XY) \,\geq\, \min\bigl\{ \mu(G),\rho(X)+\rho(Y)-1 \bigr\}, 
\] 
where \( \mu(G) \) is the submonoid-rank threshold of \( G \).
\item[(ii)]  There is a nonempty flat \( S \subseteq \textup{cl}(XY) \) and a finite-rank flat submonoid \( M \) of \( G \) such that \( SM = S \) and 
\[
\rho(X) + \rho(Y) - \rho(S) \,\leq\, \rho(M) \,\leq\, \rho(S) \,\leq\, \rho(XY).
\]
Consequently, if 
\[
\rho(XY) \,<\, \rho(X) + \rho(Y) - 1,
\]
then \( \rho(M) > 1.\)
\end{itemize}
\end{theorem}

\begin{proof}
We first prove the theorem under the additional hypothesis that \( X \) and \( Y \) are flats.  Note that \(XY\) has finite rank, by Lemma~\ref{product of finite rank}. Let \( A = \textup{cl}(XY) \) and observe that \( A \) also has finite rank, according to Proposition~\ref{closure properties}.

Choose \(y\in Y\), and define \( S_0 = Xy \) and \( R_0 = y^{-1}Y\). Then \(S_0\) and \(R_0\) are nonempty finite-rank flats, \(1\in R_0\), and
\[
S_0R_0 \,=\, (Xy)(y^{-1}Y) \,=\, XY \,\subseteq\, A.
\]
Applying Theorem~\ref{Dyson G} to \(A,S_0,R_0\), we obtain nonempty finite-rank flats \(S\) and \(R\) such that
\[
1\in R, \qquad SR \,=\, S \,\subseteq\, A,
\]
and
\begin{equation}
\label{threshold rank preservation}
\rho(S)+\rho(R) \,=\, \rho(S_0)+\rho(R_0) \,=\, \rho(X)+\rho(Y).
\end{equation}

Define \( H = \langle R\rangle \). Since \(SR\subseteq S\), Lemma~\ref{monoid lemma} gives us
\[
SH \,\subseteq\, S
\qquad\text{and}\qquad
\rho(H) \,\leq\, \rho(S).
\]
Thus \( H \) has finite rank.  Regarding the inclusion, we also have \( S \subseteq SH \), since \(1 \in H\). Hence \( SH=S \).  Let \( M = \textup{cl}(H)\) and observe that, by Lemma~\ref{flat submonoid closure}, \(M\) is a flat submonoid of \( G \) satisfying \( \rho(M) = \rho(H) \).  Thus 
\begin{equation}
\label{monoid closure rank}
\rho(M) \,=\, \rho(H) \,\leq\, \rho(S) \,\leq\, \rho(A) = \rho(XY).
\end{equation}

Next observe that \(SM=S\): by Lemma~\ref{product and closure}, we have
\[
SM \,=\, S\big(\textup{cl}(H)\big) \,\subseteq\, \textup{cl}(SH) \,=\, \textup{cl}(S) \,=\, S,
\]
and the reverse inclusion follows from \(1\in M\).

Now, since \(R\subseteq H\subseteq M\), equation \eqref{threshold rank preservation} gives 
\[
\rho(X)+\rho(Y) \,=\, \rho(S)+\rho(R) \,\leq\, \rho(S)+\rho(M).
\]
From this we obtain
\begin{equation}
\label{monoid lower bound}
\rho(M) \,\geq\, \rho(X)+\rho(Y)-\rho(S).
\end{equation}
In view of (\ref{monoid closure rank}) and (\ref{monoid lower bound}), the first assertion in (ii) is proved.

If \( \rho(XY)\geq \rho(X)+\rho(Y)-1 \), then the inequality in (i) holds and there is nothing left to do for (ii). Assume that \( \rho(XY) < \rho(X)+\rho(Y)-1 \). Then we can extend (\ref{monoid lower bound}) to
\[
\rho(M) \,\geq\, \rho(X)+\rho(Y)-\rho(S) \,\geq\, \rho(X)+\rho(Y)-\rho(XY) \,>\, 1.
\]
Since \( M \) has finite rank, it follows from the above (and Definition~\ref{submonoid rank threshold}) that \( \mu(G) \,\leq\, \rho(M) \).  From this and (\ref{monoid closure rank}), we see that the inequality in (i) again holds, and the second assertion in (ii) is proved, as well. 

Finally, we remove the assumption that \( X \) and \( Y \) are flats.  Suppose \( X,Y \) are arbitrary nonempty finite-rank subsets of \( G \), and put
\[
X_0 \,=\, \textup{cl}(X)
\qquad\text{and}\qquad
Y_0 \,=\, \textup{cl}(Y).
\]
Then
\[
\textup{cl}(XY) \,=\, \textup{cl}(X_0Y_0).
\]
Indeed, by Lemma~\ref{product and closure},
\[
X\big(\textup{cl}(Y)\big) \,\subseteq\, \textup{cl}(XY). 
\]
Using commutativity and applying Lemma~\ref{product and closure} again, we get
\[
\textup{cl}(X)\textup{cl}(Y)
\,=\,
\textup{cl}(Y)\textup{cl}(X)
\,\subseteq\,
\textup{cl}\bigl(\textup{cl}(Y)X\bigr)
\,=\,
\textup{cl}\Bigl(X\big(\textup{cl}(Y)\big)\Bigr)
\,\subseteq\,
\textup{cl}(XY).
\]
Hence
\[
\textup{cl}(X_0Y_0)\,\subseteq\,\textup{cl}(XY).
\]
The reverse inclusion follows from $XY\subseteq X_0Y_0$. 

Note, in addition, that the above and Proposition~\ref{closure properties} give
\[
\rho(X_0)=\rho(X),
\qquad
\rho(Y_0)=\rho(Y), \qquad \text{and} \qquad \rho(X_0Y_0)=\rho(XY).
\]
We can therefore apply the preceding argument to \( X_0 \) and \( Y_0 \) and then replace these sets with \( X \) and \( Y \) at the conclusion.
\end{proof}

\begin{remark}  \label{recover Cauchy davenport}
In Theorem~\ref{threshold product growth}, let us take \( G \) to be the additive group \( \mathbb{Z}/p\mathbb{Z} \), where \( p \) is a prime, and \( \mathcal{I} \) to be the universal structure. Then \( \rho(X) = |X| \) for all \( X \subseteq \mathbb{Z}/p\mathbb{Z} \), and \( \mu(G) = p \).  From part (i) of the theorem, we recover the Cauchy--Davenport inequality:
\[
|X + Y|
\,\geq\,
\min\bigl\{p, |X|+|Y|-1 \bigr\}
\]
for all nonempty \( X,Y \subseteq \mathbb{Z}/p\mathbb{Z} \).
\end{remark}

\begin{remark} \label{nontrivial submonoid gives a deficient product}
Let \( M \) and \( S \) be as in part (ii) of Theorem~\ref{threshold product growth}.  If \( \rho(M) > 1 \), then 
\[
\rho(SM) = \rho(S) < \rho(S) + \rho(M) - 1.
\]
\end{remark}

\begin{corollary}  \label{cauchy davenport equivalence}
Let \(G\) be a modular abelian independence group. Then \(G\) has the matching property if and only if, for every pair of nonempty finite-rank subsets \(X,Y\) of  \(G\),
\begin{equation*}
\rho(XY)
\,\geq\,
\min\bigl\{
\rho(G),\rho(X)+\rho(Y)-1
\bigr\}.
\end{equation*}
\end{corollary}

\begin{proof}
Assume first that \(G\) has the matching property. By Theorem~\ref{MMP}, there is no finite-rank submonoid \(H\) of \(G\) such that 
\[ 
1<\rho(H)<\rho(G). 
\] 
If \(\rho(G)\) is finite, Definition~\ref{submonoid rank threshold} therefore gives \( \mu(G)\geq\rho(G) \). If \(\rho(G)\) is infinite, any finite-rank submonoid of rank greater than one would have rank strictly smaller than \(\rho(G)\); hence no such submonoid exists and \( \mu(G)=\infty \). Thus, in either case, Theorem~\ref{threshold product growth} gives 
\[
\rho(XY)
\,\geq\,
\min\bigl\{
\rho(G),\rho(X)+\rho(Y)-1
\bigr\}
\]
for all nonempty finite-rank subsets \(X,Y\) of \( G \).

Conversely, suppose that the inequality in the statement holds and assume for a contradiction that \(G\) does not have the matching property. By Theorem~\ref{MMP}, there exists a submonoid \(H\) of \(G\) such that
\[
1<m=\rho(H)<\rho(G)
\qquad\text{and}\qquad
m<\infty.
\]
Applying the assumed inequality with \(X=Y=H\) and using the fact that \( HH=H \), we obtain
\[
m \,=\, \rho(H) \,=\, \rho(HH) \,\geq\, \min\bigl\{ \rho(G),2m-1 \bigr\}.
\]
But
\[
m<\rho(G)
\qquad\text{and}\qquad
m<2m-1,
\]
giving a contradiction.
\end{proof}

\begin{proposition}[Thresholds for coset-partition independence groups]
\label{coset partition threshold}
Let \(G\) be an abelian group and let \(H\) be a subgroup such that \(G/H\) is finite.  Equip \( G \) with the coset-partition structure \( \mathcal{I}_H \) of Proposition~\ref{coset partition proposition}. If \(G/H\) is nontrivial, then
\[
\mu(G)
=
\min\bigl\{
|Q|:Q\text{ is a nontrivial subgroup of }G/H \bigr\},
\]
so that \(\mu(G)\) is the smallest prime divisor of \(|G/H|\). If instead \(G=H\), then \(\mu(G)=\infty\).

Moreover, \(G\) has the matching property if and only if \(G/H\) is trivial or has prime order.
\end{proposition}

\begin{proof}
Let
\[
\pi:G\longrightarrow G/H
\]
be the quotient map. The image under \( \pi \) of any submonoid \( M \) of \( G \) is a subgroup \( \pi(M) \) of \( G/H \), owing to the finiteness of \( G/H \), and we have \( \rho(M) = |\pi(M)| \) by Proposition~\ref{coset partition proposition}. Furthermore, every subgroup \( Q \) of \( G/H \) is of the form \( \pi(M) \), where \( M = \pi^{-1}(Q) \) is a subgroup (and hence a submonoid) of \( G \).  It follows that 
\begin{equation} \label{submonoid subgroup equation}
\begin{aligned}
\bigl\{ \rho(M) : M\text{ is a submonoid  of }G \bigr\}  = \bigl\{ |Q| : Q\text{ is a subgroup of }G/H \bigr\}.
\end{aligned}
\end{equation}

Therefore, if \( G = H \), we clearly have
\[
\bigl\{ \rho(M) : M\text{ is a submonoid  of }G \text{ and } 1 < \rho(M) \bigr\}  = \emptyset,
\]
and so \( \mu(G) = \infty \) by Definition~\ref{submonoid rank threshold}.

Suppose that \( G/H \) is nontrivial.  Then \( |G/H| \) has at least one prime divisor, and moreover, by Cauchy's theorem, for each prime \( p \) dividing \( |G/H| \), there exists a subgroup of \( G/H \) of order \( p \).  Therefore, by Lagrange's theorem, Definition~\ref{submonoid rank threshold}, and equation~(\ref{submonoid subgroup equation}),
\[ 
\mu(G) = \min\bigl\{ |Q| : Q\text{ is a nontrivial subgroup of }G/H \bigr\} 
\]
and this value is the smallest prime divisor of \( |G/H| \).

Finally, since \( \rho(G)=|G/H| \), it follows from  Theorem~\ref{MMP} and equation~(\ref{submonoid subgroup equation}) that \( G \) has the matching property if and only if \(G/H\) has no subgroup \(Q\) satisfying
\[
1<|Q|<|G/H|.
\]
A finite group has no nontrivial proper subgroup precisely when it
is trivial or has prime order. This completes the proof.
\end{proof}

\begin{example} \label{coset threshold example}
Let \( G=C_{30} \) (cyclic group of order 30) and let \(H\) be the unique subgroup of order \(2\). Equip \( G \) with the coset-partition independence structure \( \mathcal{I}_H \). Then \( G/H\cong C_{15} \), and since the smallest prime divisor of \(15\) is \(3\), Proposition~\ref{coset partition threshold} gives
\[
\mu(G)=3.
\]
In addition, since \(G/H\) has composite order, \(G\) does not have the matching property.

This example also shows that the bound in Theorem~\ref{threshold product growth} can be attained. Let \(Q\) be the subgroup of \( G/H \) of order \(3\), and put \( U=\pi^{-1}(Q)\), where \(\pi\) is the quotient map \( G\rightarrow G/H\). Then \(U\) is a subgroup of \(G\) containing \( H \), and
\[
\rho(U)=|\pi(U)|=|Q|=3.
\]
Taking \(X=Y=U\), we obtain
\[
\rho(XY) \,=\, \rho(U) \,=\, 3
\]
and
\[
\rho(X)+\rho(Y)-1 \,=\, 5.
\]
Consequently,
\[
\rho(XY) \,=\,
\min\bigl\{\mu(G),\rho(X)+\rho(Y)-1 \bigr\} = 3.
\]

Finally, we observe that the pairing \((G,\mathcal{I}_H)\) cannot arise from the projective construction. Indeed, for every \(g\in G\), we have \( \textup{cl}(\{g\})=gH \), which has two elements, whereas every singleton is a flat in a projective independence group. 
\end{example}

\begin{example}  \label{cycle space mu and matching}
Let \( \Gamma = (\mathsf{V},\mathsf{E}) \) be a finite simple graph and let \( G = 2^{\mathsf{E}} \) be the associated cycle-space independence group from Example~\ref{cycle space independence example}. Recall that \( G \) is a coset-partition group, defined relative to the cycle space \( \mathcal{Z}(\Gamma) \). Thus \( G \) is modular. The group operation is symmetric difference, so \( G \) is Boolean.  

We will compute the submonoid-rank threshold of \( G \) and determine whether or not \( G \) has the matching property. First observe that no element of \( \mathcal{Z}(\Gamma) \) can be a singleton or doubleton, so the cosets \( X\triangle \mathcal{Z}(\Gamma) \) for \( X \in G \) with \( |X| \leq 1 \) are distinct, i.e., \( \rho(G) \geq |\mathsf{E}| + 1 \).

If \( \mathsf{E} \) is empty, then \( G \) is trivial.  Hence \( \mu(G) = \infty \) and \( G \) has the matching property.

If \( |\mathsf{E}| = 1\), then \( G \) has order 2.  Hence \( \mu(G) = 2 \) and the condition in Theorem~\ref{MMP} is vacuously satisfied, so that \( G \) has the matching property.

Suppose that \( |\mathsf{E}| \geq 2\).  Then \( \rho(G) \geq |\mathsf{E}| + 1 \geq 3 \), and since \( \{ \emptyset, \{ e \} \} \) is a subgroup of \( G \) for any \( e \in \mathsf{E} \), it follows that \( \mu(G) = 2 \) and, by Theorem~\ref{MMP}, \( G \) does not have the matching property. 
\end{example}

\bigskip

\noindent \textbf{Declarations}

\medskip
\noindent\textbf{\small Funding:} {\small The research of Mohsen Aliabadi is partially supported by an AMS--Simons Research Enhancement Grant (award no.~SFI-MPS-T-Institutes-00021934).} 

\medskip
\noindent \textbf{\small Conflict of interest:} {\small The authors declare that they have no conflict of interest.}

\medskip
\noindent \textbf{\small Data availability statement:} {\small Data sharing is not applicable to this article as no datasets were generated or analyzed during the current study.}


\end{document}